\documentclass[12pt]{amsart}
\usepackage{amsmath,amsthm}
\usepackage{amsthm}
\usepackage{latexsym}
\usepackage{amssymb}
\usepackage{amsfonts}
\usepackage{amscd}
\usepackage{booktabs}
\usepackage{slashbox}
\usepackage{dsfont} 
\usepackage{mathtools} 
\usepackage{braket} 
\input{xy}
\xyoption{all}

\newtheorem{thm}{Theorem}[section]
\newtheorem{theorem}[thm]{Theorem}

\newtheorem{corollary}[thm]{Corollary}
\newtheorem{lemma}[thm]{Lemma}

\newtheorem{proposition}[thm]{Proposition}

\theoremstyle{definition}

\newtheorem{example}[thm]{Example}
\newtheorem{remark}[thm]{Remark}
\newtheorem{observation}[thm]{Observation}

\newcommand{\R} {\mathbb{R}}
\newcommand{\Z} {\mathbb{Z}}

\newcommand{\C} {\mathbb{C}}

\newcommand{\zero}{{\overline{0}}}

\newcommand{\cI}{\mathcal {I}}

\newcommand{\cO}{\mathcal{O}}
\newcommand{\cP}{\mathcal{P}}

\newcommand{\cU}{\mathcal{U}}

\newcommand{\al}{\alpha}
\newcommand{\be}{\beta}
\newcommand{\ga}{\gamma}
\newcommand{\GAMMA}{\Gamma}

\newcommand{\la}{\lambda}

\newcommand{\fg}{\mathfrak g}
\newcommand{\fc}{\mathfrak c}

\newcommand{\fn}{\mathfrak n}
\newcommand{\fh}{\mathfrak h}
\newcommand{\fk}{\frak k}

\newcommand{\fa}{\mathfrak a}
\newcommand{\fb}{\mathfrak b}

\newcommand{\fso}{\mathfrak{so}}

\newcommand{\fm}{\mathfrak{m}}

\newcommand{\fsl}{\mathfrak{sl}}
\newcommand{\fp}{\mathfrak p}
\newcommand{\fr}{\mathfrak r}
\newcommand{\fu}{\mathfrak u}
\newcommand{\fq}{\mathfrak q}

\newcommand{\SU}{\mathrm{SU}}
\newcommand{\SO}{{\mathrm{SO}}}

\newcommand{\rEnd}{{\mathrm{End}}}

\newcommand{\Ad}{\mathrm{Ad}}
\newcommand{\ad}{\mathrm{ad}}
\newcommand{\Lie}{\mathrm{Lie}}

\newcommand{\id}{\mathrm{id}}
\newcommand{\conj}{\mathrm{conj}}
\newcommand{\Hol}{\mathrm{Hol}}
\newcommand{\Pol}{\mathrm{Pol}}

\newcommand{\tg}{g\wt}
\newcommand{\tP}{\cP\wt}

\newcommand{\lra} {\longrightarrow}

\newcommand{\wt}{\widetilde{\,\,\,}}

\newcommand{\beq}{\begin{equation}}
\newcommand{\eeq}{\end{equation}}

\begin{document}


\centerline{\bf \large Harish-Chandra highest weight representations} 

\medskip 
\centerline{\bf \large of semisimple Lie algebras and Lie groups}
\medskip 

\centerline{R. Fioresi, V.S. Varadarajan}


\begin{abstract}
  In this expository paper we describe the theory of Harish-Chandra
highest weight  representations and their explicit geometric realizations.
\end{abstract}

\section{Introduction} \label{intro-sec}

The theory initiated by Harish-Chandra in 1956, with his three seminal papers
in the American Journal of Mathematics \cite{hc} on representations of
real semisimple Lie groups, influenced researchers for the next decades and
are still a source of inspiration for many. He constructed these modules first
infinitesimally and then globally, in the space of sections on holomorphic
bundles on symmetric spaces and laid the foundation for further investigation of
unitary representations and harmonic analysis of real semisimple Lie groups
(\cite{hw}, \cite{schmidt}). His research also
paved the road to 
the successes of geometric
quantization (orbit method) and models (see \cite{kostant2}, \cite{kirillov},
\cite{gz}, \cite{gs2} and refs. therein) and later on to the SUSY
generalizations \cite{jakobsen}, \cite{cfv}, \cite{cfv2} and refs. therein.
Despite his untimely death, 
his mathematical vision had a great impact on generations of mathematicians;
for a thorough walk through his beautiful mathematical achievements,
see recollections by Howe \cite{howe-hc}, Langlands \cite{langlands-hc}
and Varadarajan et al. \cite{vsv-hc1}, \cite{vsv-hc2}.

\medskip
In the present work
we want to give a self contained exposition of the theory of the infinitesimal and global
realizations of the highest weight Harish-Chandra modules,
elucidating Harish-Chandra arguments as in \cite{hc}.

\medskip
The first step (Sec. \ref{alg-sec}) 
is the {\sl infinitesimal theory}, that is the study
of Harish-Chandra representations for the pair $(\fg,\fk)$, 
where $\fg$ is a semisimple Lie algebra and $\fk$ the complexification
of the maximal compact subalgebra of a chosen real form $\fg_0=\fk_0 \oplus \fp_0$
(real Cartan decomposition).
Notice that the index $0$ always denotes real forms and when we drop it we
mean the complexification. 
Such pair $(\fg,\fk)$ is called an {\it Harish-Chandra pair} (HC), and it is uniquely
determined by a real form $\fg_0$ of $\fg$ (see Sec. \ref{hc-pair}).

By definition, an Harish-Chandra representation of the HC pair $(\fg,\fk)$, is a representation
of $\fg$ in which $\fk$ acts {\sl finitely}, that is the sum of
irreducible representations of $\fk$ with the same character is finite
dimensional (see (\ref{hc-mod-def})).
The construction of the universal highest weight Harish-Chandra representation
$U^\lambda$, $\lambda \in \fh^*$ ($\fh$ a Cartan subalgebra of both $\fg$
and $\fk$) is based on the existence of an {\sl admissible system} for $\fg$. 
These are positive systems in which the adjoint representation of 
$\fk$ on $\fp$ stabilizes $\fp^\pm$, the sum of the
positive (negative) non compact root spaces. 
The existence of such systems is equivalent to the existence of 
center $\fc$ for $\fk$ (Prop. \ref{lemma3bis}) and to a natural invariant
complex structure on the real symmetric space $G_0/K_0$,
$\fg_0=\Lie(G_0)$, $\fk_0=\Lie(K_0)$ (Rem. \ref{cmplx-str}).
The dimension of the center is linked to the number of non equivalent
invariant complex structures on $G_0/K_0$ (see \cite{alekseevski}).
The main result for this part is Thm \ref{thm3}.
Then, based on our treatment on admissible systems, 
we proceed with the construction of the infinitesimal universal highest
weight HC module $U^\lambda$ and we give a sufficient condition for its irreducibility
(Thm \ref{thm4}).

\medskip
Once the infinitesimal theory is fully elucidated, we proceed (Sec. \ref{reps-sec}), following
\cite{hc}, to the geometric realization of the global 
Harish-Chandra representation of the real supergroup $G_0$ in the
space of holomorphic sections of a line bundle on the 
symmetric superspace $G_0 /K_0 \subset G/B$, $B$ the borel
subgroup corresponding to the fixed admissible system. 
These infinite dimensional representations of $G_0$ are the global
counterparts of the infinitesimal highest weight representations
constructed previously. We proceed as follows. 
First we consider the  Fr\'echet space of sections of the complex 
line bundle $L$ on the quotient $G/B$, 
associated with the infinitesimal character $\lambda$ (see Def. \ref{keydef}).
Then, on a neighbourhood $U$ of the origin, we define a left $\ell$
and a right $\partial$ action of $\cU(\fg)$ on $L(U)$ and establish
a duality between a subrepresentation of $L(U)$ (as left $\cU(\fg)$ module) and 
the infinitesimal HC representations studied in Sec. \ref{alg-sec} (Thm. \ref{thm5reps}).
Choosing $U=G_0B$, we can prove the main result of this part, namely
obtain the highest weight HC representation inside $L(G_0S)$, provided however that 
$L(G_0S) \neq 0$ (Thm. \ref{thm6reps}). Then we give conditions (Thm. \ref{thm9reps}) for
which this occurs. In order to do so,  
we define the {\sl Harish-Chandra decomposition} 
$P^- \times K \times P^+ \cong  P^ - KP^ + \subset G$ open, $\fp^ \pm = \Lie(P^ \pm )$, where
$P^\pm$ are abelian subgroups (see Sec. \ref{hc-cell-sec} and Lemma \ref{hc-cell-lemma}).
This is a remarkable result by itself, based on the peculiar properties of
admissible systems; $\Omega= P^ - KP^ +$ is called the {\sl Harish-Chandra big cell}.


\medskip

\bigskip
{\bf Acknowledgements.}
The authors thank Prof. C. Carmeli for many illuminating discussions.
The authors also wish to thank the UCLA Dept.
of Mathematics for the kind hospitality, during the completion of this
work. 

\section{Theory on the Lie algebra}\label{alg-sec}

\subsection{Highest weight modules}

Let $\fg$ be a semisimple Lie algebra
over $\C$. In the theory of finite dimensional representations of $\fg$, a
fundamental role is played by the highest weight modules. These are defined
with respect to the choice of a Cartan subalgebra (CSA) $\fh$ of $\fg$ and 
a positive system $P$ of
roots of $(\fg, \fk)$. 
They are parametrized by their highest weights, namely
the elements $\lambda \in \fh^*$. The universal highest weight modules are known
as Verma modules and are infinite dimensional. The irreducible highest
weight modules are uniquely determined by their highest weights and are
the unique irreducible quotients of the Verma modules. The irreducible
modules are finite dimensional if and only if the highest weight is dominant
integral, and one obtains all irreducible finite dimensional representations
of $\fg$ in this manner. These representations lift to the simply connected
complex group $G$ corresponding to $\fg$. For the basic theory of highest
weight modules see \cite{vsv1} Ch. 4, \cite{humphreys} Ch. 2,
\cite{bourbaki}.

\subsection{Harish-Chandra pairs and Harish-Chandra (HC) modules} \label{hc-pair}

In order to study 
the theory of representations for real semisimple
Lie groups, it is necessary to work in a more structured context. A
\textit{Harish-Chandra pair}
over a field $F$ of characteristic $0$ is a pair $(\fg,\fk)$, 
where:
\begin{enumerate}
\item[(i)] $\fg$ is semisimple over $F$ and $\fk$ is the 
set of fixed points of an
involutive automorphism $\theta$ of $\fg$;
\item[(ii)] $\fk$ is reductive in $\fg$. 
\end{enumerate}
This implies that $\fk$ is reductive and
$$
\fk = \fk' \oplus \fc, \quad
\fk' = [\fk, \fk], \qquad
\fc = \, \hbox{center of}\,\, \fk.
$$
(refer to \cite{knapp}, Ch. VI and \cite{helgason}
for the theory of Cartan involutions
and their complexifications).
In what follows we work over $\C$. Since $\theta^2 = 1$, 
$\fg$ is the direct sum of
the subspaces where $\theta = \pm 1$. We write $\fp$ 
for the eigenspace  of $\theta$ for the eigenvalue $-1$,
$\fk$ by, its very definition, being the eigenspace of
eigenvalue $1$.
Since the Cartan-Killing form is invariant under all
automorphisms of $\fg$ we have:
$$
\fp = \fk^\perp, \qquad
\fg = \fk \oplus \fp
$$
where $\perp$ refers to the Cartan-Killing form. 
The restrictions of this form to
$\fk$ and $\fp$ are therefore both non-degenerate.
The fact that $\theta$ is an involutive
automorphism implies that
\beq \label{key-cd}
[\fk, \fk] \subset \fk, \qquad
[\fk, \fp] \subset \fp,\qquad
[\fp,\fp] \subset \fk.
\eeq
These properties of $(\fg, \fk)$ 
 lead to calling the pairs $(\fg, \fg_\zero)$, in the
 theory of Lie super algebras and Lie supergroups,
 {\sl super Harish-Chandra pairs} \cite{ccf}, \cite{cfv}
 ($\fg$ a Lie superalgebra, $\fg_\zero$ its even part).
 Nevertheless in this note, we limit
 ourselves to the ordinary theory of Lie
 algebras, and we shall use, as customary, the suffix $0$ to
 indicate a real form of a Lie algebra or a Lie group.
Let $G_0$ be a real connected semisimple group having finite center with
Lie algebra $\fg_0$. Let $K_0$ be a maximal compact subgroup of $G_0$ with Lie
algebra $\fk_0 \subset \fg_0$. Then we say that $(\fg_0 , \fk_0 )$ 
is a \textit{Harish-Chandra pair (HC pair) over $\R$}; its
complexification $(\fg, \fk)$ is a HC pair over $\C$.
In this case, there is a unique
involution of $\fg_0$ such that $\fk_0$ is its set of fixed points, 
called the \textit{Cartan
involution} (see \cite{knapp} Ch. VI).

However, it must be noted that there are involutions of $\fg_0$
which are not Cartan involutions,
{which also can account for HC pairs}. For example,
for $\fg_0 = \fsl(n, \R)$,
 we can take $\theta$ to be
$X \mapsto -F X^t F$, where $F$ is the matrix:
$$
F =\begin{pmatrix} I_p & 0 \\ 0 & I_q \end{pmatrix}, \qquad
n = p + q, \, \,p,q \leq 1.
$$
Then $\fk_0 = \fso(p, q, \R)$ and $(\fg_0 , \fk_0 )$ is also a 
HC pair over $\R$, 
but in this case $\fk_0$ corresponds
to a non-compact subgroup of $G_0$.
The HC pairs arising from $(G_0 , K_0 )$, 
{$K_0$ compact},
are our main interest, although much of the theory can be formulated in
the more general
context of HC pairs over $\C$. When we write $(\fg, \fk)$, it means the HC
pair is defined over $\C$; the suffix 0 indicates that it is defined over $\R$.

For a Harish-Chandra pair $(\fg,\fk)$, we have the class of $(\fg,\fk)$-modules,
which are $\fg$-modules $V$ such that $V$ splits as an algebraic direct sum of
finite dimensional irreducible $\fk$-modules. In this case we can write
$$
V = \oplus_{\theta \in E(\fk)} V_\theta
$$
where $E(\fk)$ is the set of equivalence classes 
of irreducible finite dimensional
representations of $\fk$ and $V_\theta$ is the span of all 
$\fk$-irreducible subspaces of $V$
belonging to the class $\theta$. The module $V$ is called a
\textit{Harish-Chandra
  module} if:
\beq\label{hc-mod-def}
\hbox{dim}(V_ \theta ) < \infty
\qquad \hbox{for all} \,  \theta.
\eeq
The significance of this class of modules can be seen from the
following example. Suppose that the Harish-Chandra pair $(\fg,\fk)$ arises from
a pair $(G_0 , K_0 )$, i.e. by complexifying the
Lie algebras. Then, irreducible $(\fg,\fk)$-modules, which are HC-modules,
are the objects of interest from the point of view of representations of
the group $G_0$ because of the following fact: if $H$ is an irreducible Banach
space representation of $G_0$ and $V$ is the set of $K_0$-finite vectors
(see below), then $V$
is an irreducible HC-module. The converse is also true: every irreducible
HC-module arises in this manner.
The generic HC module is not of the highest weight type, but under
certain circumstances highest weight modules are also HC modules and it
is the purpose of these notes to investigate these closely.

\subsection{Highest weight $(\fg,\fk)$-modules when 
  rank of $\fg$=rank of $\fk$}

If $V$ is a $\fg$-module and $\fa \subset \fg$ is a Lie subalgebra, 
a vector $v \in V$ is $\fa$-finite,
if $v\in W$ for some finite dimensional $\fa$-stable subspace $W$. Let 
$V [\fa]$ be
the subspace of all $\fa$-finite vectors. Then $V [\fa]$ is a 
$\fg$-submodule of $V$;
this follows from the easily proved fact that if $W$ is finite dimensional and
$\fa$-stable, then $\fg[W ]$ is again finite dimensional and $\fa$-stable.
Let us now assume that:
$$
(\fg,\fk) \quad \hbox{is a HC pair and rk }\fg = \hbox{rk} \, \fk.
$$
Then we can choose a Cartan subalgebra (CSA) $\fh$ so that
$$
\fh \subset \fk \subset \fg
$$
and $\fh$ will be a CSA of both $\fk$ and $\fg$. 
We fix a positive system $P$ of roots
for $(\fg, \fh)$ and write $\alpha >0$ interchangeably with $\alpha \in P$. 
We are interested
in highest weight modules (with respect to $P$) which are also HC modules.

\begin{lemma}
If $U$ is a highest weight module with a highest weight vector
$u$, the following are equivalent:
\begin{enumerate}
\item[a)] dim$(\cU (\fk)u)< \infty$.
\item[b)] $U$ is $(\fg,\fk)$-module
\item[c)] $U$ is a HC module.
\end{enumerate}

If these are satisfied, $\cU (\fk)u$ is an irreducible $\fk$-module.
\end{lemma}

\begin{proof}
  Since $\fk$ is not semisimple, one must be a little careful.
Let $\fc$ denote the center of $\fk$.
  For instance,
a finite dimensional $\fk$-module is fully reducible if and only 
if the action of
$\fc$ is completely reducible. In the present case, since 
$\fc \subset \fh$ and $U$ is a direct
sum of weight spaces, it follows that $\fc$ acts semisimply on $U$. 
In particular,
any finite dimensional submodule for $\fk$ is fully reducible. If $u$ is 
$\fk$-finite,
then $u \in U [\fk]$. As $U [\fk]$ is a $\fg$-module, we see that,
$U [\fk] = U$. 
So a) $\implies$ b).
c) $\implies$ a) trivially. If we assume b), we must prove that the spaces 
$U_\theta$ are
finite dimensional. If not, and if $\mu$ is a weight of $\theta$, 
then $\mu$ occurs with
infinite multiplicity, so that dim $U [\mu] = \infty$, a contradiction. 
So b) $\implies$ c).
Now $\cU (\fk)u$ is a highest weight module for $\fk$ 
of finite dimension on which
$\fc$ acts through scalars, namely, $Cw = \lambda(C)w$ where 
$C \in \fc, w \in \cU (\fk)u$
and $\lambda$ is the highest weight. 
Hence $\cU (\fk)u = \cU (\fk')u$, 
($\fk'=[\fk,\fk]$),
  and so it is a finite
dimensional highest weight module for the semisimple algebra $\fk'$. Hence
it is irreducible.
\end{proof}


Since $\fh \subset \fk$, both $\fk$ and $\fp$ are
stable under ad $\fh$, and as the root spaces are one-dimensional, each root
space $\fg_\alpha$ is contained either in $\fk$ or $\fp$; 
we then refer to $\alpha$ as a \textit{compact}
or \textit{non-compact} root respectively. 

Write $P_k$, $P_n$ for the set of compact
and non-compact roots in $P$. If $\alpha, \beta, \alpha + \beta$ 
are roots, then the relation
$[\fg_\alpha , \fg_\beta ] = \fg_{\alpha+\beta}$ 
implies the following: if $\alpha, \beta$ are both compact or both
non-compact, then $\alpha + \beta$ is compact, 
while if one of them is compact and
the other non-compact, then $\alpha + \beta$ is non-compact;
{this is a straighforward consequence of (\ref{key-cd}).}
We are interested in determining the highest weight modules for
$(\fg, \fh, P )$ which are $(\fg,\fk)$-modules.
{Let $V$ be one of such modules, of highest weight $\la \in\fh^*$
  and $v \in V$ a highest weight vector.} 
Since $\cU (\fk)v$ is a finite dimensional highest weight module 
for $(\fk, \fh, P_k )$, it follows that the highest weight $\lambda$ must
be dominant integral for $P_k$, i.e., we must have
$$
\lambda(H_\alpha ) \in \Z_{\geq 0},  \quad (\alpha \in P_k)
$$

\medskip
We shall now show that $\lambda$ must satisfy additional conditions. 

Let $\fp^\pm$ be
the span of the root spaces corresponding to the roots 
$\beta \in \pm P_n$. Neither of
these is in general stable under the adjoint action of $\fk$, 
but they may admit non-zero subspaces stable under ad $\fk$; 
since $\fh \subset \fk$, such subspaces are spans
of root spaces for non-compact roots. 
A root $\beta$ is said to be \textit{totally positive}
if $\fg_\beta$ is contained in a subspace $\fm$ of $\fp^+$ stable under $\fk$; 
if $\fm = \oplus_{\gamma\in R} \fg_\gamma$,
then $\beta \in R$ and all roots in $R$ are 
also totally positive. Negatives of totally
positive roots are called \textit{totally negative}. 
Their behavior is similar to the
totally positive roots because of the fact that there is an automorphism
of $\fg$ that is $-\id$ on $\fh$. Such an automorphism will 
take $\fg_\alpha$ to $\fg_{-\alpha}$ for all
roots $\alpha$, in particular preserving $\fk$ and $\fp$.
We write $P_t$ for the set of totally positive roots and
$$
\fp^\pm_t = \oplus_{\beta\in P_t} \fg_{\pm\beta}
$$
Obviously $\fp_t$  is the largest ad $\fk$-stable subspace of 
$\fp$; it may be $0$. We
say that the positive system $P$ is \textit{admissible} if
$$
\fp^\pm_t\neq 0, \qquad 
\hbox{i.e.}, \qquad \pm P_t \neq 0.
$$
For example in $A_2$, with root system $\Delta=\pm\{\al,\be,\al+\be\}$,
and compact roots $\pm\al$, we have that $P=\{\al,\be, \al+\be\}$
is admissible ($P_k=\{\al\}$, $P_n=\{\be, \al+\be\}$), while 
$P'=\{\al+\be,-\be,\al\}$ is not admissible ($P'_k=\{\al\}$, $P'_n=
\{\al+\be,-\be\}$), see \cite{knapp} Ch. VII, \cite{kostant} and also
\cite{cf}, \cite{df} for generalizations.

\medskip
Our aim is to prove the following theorem, which reveals the significance
of total positivity for the problem of constructing infinite dimensional
highest weight $(\fg,\fk)$-modules.

\begin{theorem} \label{thm1}
Let $V$ be a non-zero highest weight $(\fg,\fk)$-module with respect
to a positive system $P$ of roots of $(\fg, \fh)$, 
of highest weight $\lambda$. Then $\lambda(H_\gamma ) \in \Z_{\geq 0}$
for all positive roots $\gamma$ which are not totally positive.
In particular, if $P$ is not admissible and $V$ is irreducible, then $V$ 
is finite
dimensional.
\end{theorem}

The proof of this theorem is quite delicate and depends on the following
lemmas. Before stating them we need some preparation.

\medskip
Let $v$ be a highest weight vector. For any positive root $\gamma$ 
let $\fa_\gamma =\fh \oplus \fg_\gamma \oplus \fg _{-\gamma}$. Let 
$Q$ be the set of all positive roots such that $v$ is $\fa_\gamma$-finite
and let $W_Q$ be the subgroup of the Weyl group generated by the reflections
$s_\gamma$ for $\gamma \in Q$. Let us also write $W_k$ 
for the Weyl group of $\fk$. Since the
vectors $w \in V$, such that $w$ is $\fa_\gamma$-finite, 
form a $\fg$-module, $Q$ is also the set
of all roots $\gamma > 0$ such that $v$ is $\fa_\gamma$-finite. 
It is known that $v$ is $\fa_\gamma$-finite if and only if
$X^rv =0$ for $r >>0$, and that, in this case, $\lambda(H_\gamma ) \in 
\Z_{\geq 0}$. If $\gamma \in Q$, we can split 
$V$ as a direct sum of finite dimensional
irreducible $\fa_\gamma$-modules from which we conclude that the 
set of weights of $V$ is stable under $W_Q$. Since $v$ is $\fk$-finite, 
$P_k \subset Q$ and so $W_k \subset W_Q$.
Suppose $\beta$ is totally positive. Then, there is a subset $R$ 
of $P$ such that
$R$ contains $\beta$ and $\oplus_{\gamma\in R} \fg_\gamma$ 
is stable under the adjoint action of $\fk$. Hence
$R$ is stable under $W_k$, showing that $s\beta >0$ 
and in fact totally positive for
all $s \in W_k$. The lemma below goes in the other direction.

\begin{lemma} \label{lemma2}
Let $\gamma$ be a positive non-compact root. If $s\gamma < 0$ for some
$s \in W_Q$, then $\gamma \in Q$.
\end{lemma}

\begin{proof}
The set of weights of $V$ is stable under $W_Q$ and so $s\lambda$ is a weight
of $V$. Hence 
$$
\lambda - s\lambda = \sum_{1 \leq i \leq \ell} n_i\be_i, \qquad
\hbox{($n_i$ are integers $\geq 0$)}
$$
where $\{\beta_1 , \dots, \beta_\ell \}$ 
is the set of simple roots in $P$. We shall show that
$X^p_{-\gamma}v=0$ if 
$p > \mathrm{max}_ i n_i$. 
Suppose for some integer $p\geq 1$, we have $X^p_{-\gamma}v\neq 0$.
Then $\lambda- p\gamma$ is a weight of $V$. Write $s\gamma =
 -\beta$ where $\beta >0$.
Then $s\lambda + p\beta$ is a weight of $V$ so that
$\lambda-s\lambda-p\beta= \sum_{1 \leq i \leq \ell}m_i\be_i$
the $m_i \in \Z_{\geq 0}$.
Then $p\beta =  \sum_{1 \leq i \leq \ell}(n_i-m_i)\be_i$; 
as $\beta >0$, we can also
write $\beta = \sum_i k_i \beta_ i$ where the $k_i \in \Z_{\geq 0}$
$k_{i_0} >0$ for some $i_0$.
We have $pk_{i_0} = n_{i_0} - m_{i_0} \leq n_{i_0}$ giving $p \leq pk_{i_0} 
\leq  n_{i_0}$.
\end{proof}

\begin{lemma} \label{lemma3} If $\gamma$ is a positive root that 
vanishes on the center $\fc$ of $\fk$, then
$s \gamma < 0 $ for some $s \in W_k$ and hence $\gamma \in Q$. 
If $\gamma$ is totally positive, it
cannot vanish on $\fc$ and all $s\gamma > 0$, 
even totally positive, for $s \in W_k$. In
particular $\fc\neq 0$ when totally positive roots exist.
\end{lemma}

\begin{proof}
Let $\delta = \sum_{s\in W_k}s\gamma$ ($\delta$ need not be a root). 
Since $\fc$ is fixed elementwise by $W_k$, we see that all 
$s\gamma$ vanish on $\fc$ and so $\delta$ must be $0$ on
$\fc$. On the other hand, $\delta$ is fixed by all elements of
 $W_k$ and so $\delta(H_\theta ) =0$,
for all $\theta \in P_ k$. So $\delta$ must be 0 since the 
$H_\theta$ and $\fc$ span $\fh$. But then a sum
of positive roots cannot be $0$ and so $s\gamma$ must be $<0$ 
for some $s \in W_k$.
By Lemma \ref{lemma2}, we conclude that $\gamma \in Q$. If $\gamma$ 
is totally positive, then $\fg_\gamma$ is
contained in a sum $\fm$ of root spaces contained in $\fp^+$ and stable 
under $\fk$, so that all the root spaces $\fg_{s\gamma}$ are contained
in $\fm$; this shows that all $s\gamma >0$. The
previous argument shows that $\gamma$ cannot be 0 on $\fc$.
\end{proof}

Now we go to the proof. of Theorem \ref{thm1}.

\begin{proof} 
Let $v$ be a highest weight vector. We want to
prove that if $\gamma >0$ is not in $Q$, 
then $\gamma$ is totally positive. Since we know
that $V$ is $\fk$-finite, we have $P_k \subset Q$ 
and so we may assume that $\gamma$ is a
non-compact positive root.
Let $\fm$ be the minimal ad $\fk$-stable subspace of $\fp$ 
containing $\fg_ \gamma$. Now
$\fm$ is spanned by the non-compact root spaces, and as these are all one-
dimensional, it follows from the complete reducibility of the action of $\fk$
on $\fp$ that $\fm$ is irreducible. 
Let $\beta_0$ and $\beta_1$ be the highest and lowest
roots belonging to $\fm$. If $\beta_1 >0$ all the roots belonging 
to $\fm$ will be $>0$, showing that 
$\gamma$ must be totally positive. Hence it suffices to prove
that $\beta_ 1 >0$. Suppose to the contrary that $\beta_ 1 <0$. Now there is an
element $s \in W_k$ such that $s\beta_0 = \beta_1 <0$. So by Lemma 
\ref{lemma2}, we know
that $\beta_0 \in Q$. Let $\beta$ be a maximal positive root belonging to
$\fm$, but which is not in $Q$. Clearly $\beta \neq \beta_0$. 
Then there is a positive compact
root $\alpha$ such that $[\fg_\alpha , \fg_\beta ] = \fg_{\alpha+\beta}
\neq 0$ and so $\alpha + \beta$ is a root, which is positive.
 Moreover $\alpha + \beta \in Q$. Clearly $\alpha + \beta$ 
is positive non-compact. Let
$\theta = \beta + (\alpha + \beta) = 2\beta + \alpha$. 
We claim that $\theta$ is not a root. If it were a root,
it must be compact, and so $\theta$ and $\alpha$ are both compact roots. 
Hence they
both vanish on $\fc$, from which we infer that 
$\beta = (1/2)(\theta - \alpha)$ must also
vanish on $\fc$. By Lemma \ref{lemma3} we have that 
$\beta \in Q$, a contradiction. Hence
$\beta + (\alpha + \beta)$ is not a root but 
$\beta - (\alpha + \beta) = -\alpha$ is a root. Let $t\geq 1$ be the
largest integer such that $\beta - t(\alpha + \beta)$ is a root. Then
$s_{\alpha+\beta} (\beta - t(\alpha + \beta)) = \beta$
or, equivalently,
$$
s_{ \alpha+\beta} \beta = \beta - t(\alpha + \beta) = -\alpha - 
(t - 1)(\alpha + \beta) < 0
$$
showing that $s_{\alpha+\beta} \beta <0$. But 
$\alpha + \beta \in Q$ and so, by Lemma
\ref{lemma2}, $\beta \in Q$, a contradiction.
This finishes the proof of the theorem.
\end{proof}


\subsection{Structure of the set of totally positive roots.}

\medskip
Recall that $P_n$ denotes
the set of positive non compact roots in the positive
system $P=P_k \cup P_n$, while $\theta$ is the Cartan involution
of the complex semisimple Lie algebra $\fg$.

\medskip
The basic result we want to prove is the following.

\begin{theorem} \label{thm2}
Let
$$
\fg_t = \fk_ t \oplus \fp_ t, \quad \hbox{where} \quad
\fk_ t = [\fp_ t , \fp_ t ], \quad
\fg_1 = \fg^\perp_t.
$$
Then $\fg_t$, $\fg_1$ are ideals of $\fg$ which are $\theta$-stable, and
$$\fg= \fg_t \oplus \fg_1$$
Moreover $P_t$ is precisely the set of positive non-compact roots of $\fg_t$ 
and $P_n \setminus P_ t$ 
is precisely the set of positive non-compact roots of $\fg_1$.
\end{theorem}

Before its proof we need a lemma.

\begin{lemma} \label{lemma4} 
Let the notation be as above.
Let $\fq_1$, $\fq_2$ be two subspaces of $\fp$ stable under ad $\fk$. Suppose
that $\fq_1 \perp \fq_2$. Then $[\fq_1 , \fq_2 ] = 0$.
\end{lemma}

\begin{proof}
Let $B$ be the Cartan-Killing form of $\fg$. If $X \in \fk$, $Y \in
\fq_1$, $Z \in \fq_2$, then $B(X, [Y, Z]) = B([X, Y ], Z) =0$ 
because $[X, Y ] \in \fq_1$, $Z \in
\fq_2$.
\end{proof}

Now we go to the proof of Thm \ref{thm2}.

\begin{proof}
Since $\fp_t$ is stable under ad $\fk$, it is immediate that $\fg_t$ is stable
under ad $\fk$. It is also obvious that 
$\fg_t$ is stable under $\fp_t$. We must show
that it is stable under $\fp$. 
Let $\fq = \fp^\perp_t$.
Then $\fq$ is also stable under ad $\fk$. By 
Lemma \ref{lemma2} we have $[\fq, \fp_t ] = 0$, hence
also $[\fq, \fk_ t ] =0$, since $\fk_ t = [\fp_t , \fp_t ]$. 
So $[\fq, \fg_t ] =0$. To prove that $\fg_t$ is stable
under ad $\fp_t$ is thus enough to show that $\fp = \fp_t \oplus \fq$, 
or equivalently, $\fq \cap \fp_t =0$. If 
$\fg_\beta$ belongs to this intersection, then from $\fg_\beta \in \fp_t$, we get
$\beta = c\gamma$ where $c = \pm 1$ and $\gamma \in P_t$. 
On the other hand, as $\fg_\beta \in \fq$, $\fg_\beta$ is
orthogomal to $\fg_{\pm\gamma}$. 
Hence $\fg_\beta \perp \fg_{\pm\beta}$. This is impossible. Thus we have
proved that $\fg_t$ is an ideal. It is clearly 
$\theta$-stable and so $\fg_1$ is stable under $\theta$
and is the ideal complementary to $\fg_t$. The remaining assertions are now
obvious.
\end{proof}

\begin{remark}
The condition $\fk_ t = [\fp_t , \fp_t ]$ means that $\fg_t$ 
has no ideal factors
which are ``compact''.
\end{remark}

Because of this result, we can direct our attention to the case when
$P_n = P_ t$, i.e., $\fg= \fg_t$; hence
all non compact positive roots are totally positive.
Since $P_n = P_ t$, we know that $\fp^+_t$
 is stable under ad $\fk$ and multiplicity
free (as the root spaces are one dimensional). 
Hence we have a unique (up
to ordering) decomposition 
$$
\fp^+_t = \oplus_{ 1\leq i\leq s} \, \fp_i \qquad (1 \leq i \leq s) 
$$
into irreducible
modules for $\fk$. Define $\fp^-_i$
so that the roots belonging $\fp_i^-$ are the negatives
of the roots belonging to $\fp_i^+$. 
Let $\beta_ i$,  $(1 \leq i \leq s)$ be the lowest root of $\fp^+_i$.

\begin{theorem}\label{thm3} 
  Let $P_n = P_t$  
  and 
let $\alpha_ 1 , \dots , \alpha_r$ be the simple roots of $P_ k$.
Then, 
$\{\alpha_1 , \dots , \alpha_r , \beta_ 1 , \dots , \beta_ s \}$ 
is the set of simple roots of $P$. Moreover we
have
\begin{enumerate}
\item[(a)] $r + s =$ rk $\fh$.
\item[(b)] $s$ = dim($\fc$) where $\fc$ is the center of $\fk$, 
and the restrictions of the
$\beta_ i$ to $\fc$ are linearly independent.
\item[(c)] If $\beta$ is a non-compact positive root, 
there is exactly one $i$, $(1 \leq i \leq
s)$ and integers $m_ j\geq 0$ such that 
$$
\beta = \beta_ i + m_ 1 \alpha_1 + \dots + m_ r \alpha_r 
$$
$i$, $m_ j$ are uniquely determined by $\beta$. 
In particular $[\fp^+ , \fp^+ ] =0$.
\item[(d)] $\fg$ decomposes as the sum of $s$ ideals 
$\fg_i = \fk_ i \oplus \fp_ i$ which are $\theta$-stable. 
Each of these have the property that all non-compact roots
are totally positive, $[\fp_ i , \fp_ i ] = \fk_ i$, 
and the dimension of the center
of $\fk_ i$ is $1$. In particular each $\fg_i$ is simple.
\item[(e)] The integer $s$ 
is also the number of irreducible components of $\fp^+_t$
as a $\fk$-module. In particular, $\fg$ is simple if and only if $\fp^+_t$
is irreducible.
\end{enumerate}
\end{theorem}

\begin{proof} 
  Let $\beta_ i$ be the lowest root (i.e. weight) of $\fp^+_i$. Then, for any
  non-compact root $\beta$,
  there exits a unique $i$ such that $\fg_\beta \subset 
\fp^+_i$, and $\fg_\beta$ can
be reached by applying positive compact root vectors to $\fg_{\beta_ i}$. 
Hence, $\beta = \beta_ i + m_1 \alpha_1 + \dots m_r \alpha_r$, where 
the $m_j$ are integers $\geq 0$. Let $S =
\{\alpha_1 , \dots , \alpha_r , \beta_ 1 , \dots , \beta_ s \}$. 
Then, every positive root is a non-negative integral linear combination 
of elements of $S$. So $r + s\geq \ell$, where $\ell$ is the
rank of $\fg$. On the other hand, suppose $H \in \fh$ 
is an element such that all
elements of $S$ vanish at $H$. Then, $H$ must centralize $\fk$ and $\fp$, 
hence also $\fg$.
So $H =0$, showing that the elements of $S$ are lineary independent. The
same argument shows also that the restrictions to $\fc$ of the 
$\beta_ i$ are linearly
independent. So $S$ is the set of simple roots in $P$ and $i$ and the $m_j$ are
unique. The formula for the non-compact roots shows that the sum of two
elements of $P_t$ is never a root. The proofs of (a), (b), (c) are now clear.

We now take up (d) and (e). Let $\fp_ i = \fp^+_i \oplus \fp_ i^-$. 
Clearly the $\fp_ i$ are
mutually orthogonal and $\fp = \oplus_ i \fp_ i$. 
Let $\fk_ i = [\fp_ i , \fp_ i ]$ and $\fg_i = \fk_ i \oplus \fp_ i$. Then
$\fg_i$ is stable under $\theta$ as well as ad $\fk$ and ad $\fp_ i$. 
We now claim that for $i \neq j$,
$[\fp_ i ,\fp_j ] =0$. Certainly $\fp_i \perp\fp_j$, because, if 
$\beta_ r \in\fp_r$, then $\beta_ i \pm \beta_ j \neq 0$ proving
the claim in view of Lemma \ref{lemma4}. 
Hence $\fg_i$ is stable under all ad $\fp_j$ $(j \neq i)$,
hence under ad $\fp$. So the $\fg_i$ are ideals in $\fg$. 
The results of (a)-(c) now give
(d) and (e).
\end{proof}


We also have the following result we use in the sequel.

\begin{proposition} \label{prop1}
Let the notation be as above.
\begin{enumerate}
\item[(a)] If $\beta$ is a totally positive root, then any root of the
form $\gamma = \beta + \sigma$, 
where $\sigma$ is an integral linear combination of the compact
roots is totally positive and lies in the irreducible ad $\fk$-module 
generated
by $\fg_\beta$. 
\item[(b)] If $\gamma$ is totally positive and $\alpha$ is a compact root, 
the $\alpha$-chain
containing $\gamma$ has length $\leq  3$. 
\item[(c)] If $\fg$ is simple then either no root is
totally positive or all positive roots are totally positive. 
\item[(d)] If $\fg= \fg_t$ is
simple, then $P$ and $P^- = P_k \cup -P_ n$ 
are the only two admissible positive
systems containing $P_k$ a fixed positive
system for $\fk$. Otherwise, if $\fg_i$, $(1 \leq  i \leq  s)$ 
are the simple ideals of $\fg$, and $\fp_i^\pm$ 
are as defined earlier, then the number of
admissible possible systems containing $P_k$ is $2^s$.
A non-compact root belongs,
to exactly one of 
$\fp^\pm_i$.
\end{enumerate}
\end{proposition}

\begin{proof} (a) We have the ideal decomposition $\fg= \fg_t \oplus \fg_1$ 
and $\beta$ is a root
of $\fg_t$. If $\gamma$ is a root of $\fg_1$, 
it will vanish on $\fc \cup \fg_t$ and so $\beta$ will also vanish
there, contradicting total positivity. Thus $\gamma$ is also a root of 
$\fg_t$. Let us
write $\sigma$ $(\sigma^+ )$ with or without suffixes 
for integral (positive integral) linear
combinations of the simple compact roots of $\fg_t$. Then, using the notation
of Theorem \ref{thm3}, we have 
$\beta = \beta_ i + \sigma_1^ +$ so that $\gamma = \beta_ i + \sigma_2$ 
on the one hand
and $\gamma = \pm(\beta_ j +\sigma_3 ^+ )$ on the other hand, 
depending on whether $\gamma$ is positive
or negative. Hence $\beta_ i = \pm(\beta_ j + \sigma_3 ^+ )$. 
Restricting to the center of $\fk_t$  and
remembering that the restrictions of the $\beta_ m$ are
linearly independent, we
see that we have to take the plus sign and $j = i$. The conclusions of (a)
now follow at once.

\medskip
(b) We can take the $\alpha$-chain to be 
$\{\gamma - p\alpha\}$, $(p =0, 1, \dots , k)$ where
$\gamma - k\alpha = s_ \alpha \gamma$ so that $k = \gamma(H_\alpha )$. 
It is a question of proving that $\gamma(H_\alpha ) \leq 2$. 
Suppose $\gamma(H_\alpha )\geq 3$. Then $\alpha(H_\gamma ) >0$ and so $\geq 1$, 
showing that $m =
\alpha(H_\gamma )\gamma(H_\alpha )\geq 3$. Consider the root 
$\beta = s_ \gamma s_ \alpha \gamma = (m - 1)\gamma - \gamma(H_\alpha )\al$.
Since $m - 1\geq 2$ this contradicts (c) of Theorem \ref{thm3}.

\medskip
(c) If $\fg$ is simple, then $\fg= \fg_t$ or $\fg= \fg_1$.

\medskip
(d) Let $P'$  be an admissible positive system. Then all roots of $P'_n$ 
belong to $\fg_t$. We may the assume that $\fg= \fg_t$ and is simple. 
If $P_t'$  contains
an element from $P^\pm_n$, then by the irreducibility of
$\fp^\pm$ under $\fk$ we see that $P_t'$ 
must contain all of $P^\pm_n$. Thus $P^\pm$ are the only 
admissible positive systems
containing $P_k$. 
In the general case, let $\epsilon = (\epsilon_ i )$ be an $s$-tuple of signs 
$\pm 1$
and let $\fq^ \epsilon = \oplus \fp_ i^ {\epsilon_i}$. Let 
$P^\epsilon$ be the set of roots belonging to $\fq^ \epsilon$. We claim
that $P^ \epsilon$ is an admissible positive system. 
To prove that it is a positive
system, it is enough to find a point in $\fh$ 
at which all the elements of $P^ \epsilon$
are $>0$. Write $\beta_i^ + = \beta_ i$ for the lowest root in $\fp^+_i$; 
if $\gamma_i$ is the highest
root of $\fp^+_i$,
 then $\beta_i = -\gamma_i$ is the lowest root of $\fp_i$, 
and it is a question of finding a point of $\fh$ 
at which all of $P_k$ and all $\beta_i^{ \epsilon_i}$ are $>0$. 
Since the
$\beta_i$ are linearly independent when restricted to 
$\fc$ we can find a $C \in \fc$ such
that $\beta_i^{ \epsilon_i} (C) >0$ for all $i$. 
On the other hand, we can find a $U$ in the span
$\fh'$ of the $H_\alpha$ for the 
compact roots such that $\alpha(U ) >0$ for all $\alpha \in P_k$.
Then, for $H = C + \eta U$ for sufficiently small $\eta >0$ 
has the property that
$\beta(H) >0$ for all $\beta = \alpha \in P_k$, 
$\beta = \beta_i^{ \epsilon_i}$. Obviously there are no other
admissible positive systems containing $P_k$.
\end{proof}

\subsection{Harish-Chandra homomorphism} \label{hc-homo}

Let $\zeta$ be the center of
$\cU (\fg)$ and $\cU [0]$ the subalgebra of $\cU (\fg)$ commuting with
$\fh$. More generally,
for any $\mu \in \fh$, let $\cU [\mu]$ be the subspace of
$\cU (\fg)$ given by
$$
\cU [\mu] = \{a \in \cU (\fg)\, |\, [H, a] = \mu(H)a, \,\, \hbox{ for all}
\,\, H \in h\}
$$
Then $\cU [0]$ is a subalgebra, $\zeta \subset \cU [0]$,
and ($\cU [\mu]$) is a grading of $\cU(\fg)$; moreover
$\cU [\mu] \neq 0$ if and only if $\mu$ is in the
$\Z$-span of the roots (the root lattice).
If $\gamma_ 1 ,\dots, \gamma_ t$ is an enumeration of
the positive roots and $(H_ i )$ is a basis for $\fh$,
then elements of $\cU [0]$ are linear 
combinations of 
$$
X_{-\gamma_1}^{n_1}
\dots H_ 1^{ c_ 1} \dots 
X_{ \gamma_1}^ {p_ 1}  \dots \quad 
\hbox{with} \quad (p_ 1 - n_ 1 )\gamma_ 1 + \dots = 0.
$$ 
It is
then clear that every term occurring in such a linear combination must
necessarily have some $p_ i > 0$ except those that are just monomials in the
$H_ i$ alone. So for any 
$u \in \cU [0]$ we have an element $\beta_ P (u) = \beta(u) \in \cU (\fh)$
such that
\beq\label{hchomom}
u \equiv \beta(u)\, (\hbox{mod} \cP), \qquad
\cP :=
\cU (\fg)\fg_ \gamma,\qquad
\gamma\in P
\eeq
The action of $u$ on the Verma module $V_ \lambda$ must leave the weight spaces
stable since it commutes with $\fh$, and so applying it to the highest weight
vector $v_ \lambda$ we see that $uv_ \lambda = \beta(u)(\lambda)v_ \lambda$ 
where we are identifying $\cU (\fh)$
with the algebra of all polynomials on $\fh^*$ so that 
$\beta(u)(\lambda)$ makes sense. It
follows from this that, if $u \in \cU (\fh) \cap \cP$, 
then $u(\lambda) = 0$ for all $\lambda$ and so $u = 0$,
i.e., $\cU (\fh) \cap P = 0$. Hence $\beta(u)$ 
is uniquely determined by the equation
(\ref{hchomom}), and the map $u \mapsto \beta(u)$ is a homomorphism of 
$\cU [0]$ onto $\cU (\fh)$.
Since $\zeta\subset \cU [0]$, 
we thus have a homomorphism of $\zeta$ into $\cU [\fh]$. This
is the \textit{Harish-Chandra homomorphism} (see \cite{knapp} Ch. VII).
Harish-Chandra proved that $\beta$ is
an isomorphism of $\zeta$ onto the algebra of all elements of $\cU (\fh)$ 
invariant
under a certain (affine)
action of the Weyl group $W$ on $\fh$. More precisely,
let $\delta = \delta_ P = (1/2) \sum_{\alpha\in P} \alpha$, and for 
$s \in W$ let $s_ A \lambda = s(\lambda + \delta) - \delta$.
Then $s \mapsto s_ A$ is an (affine) action of $W$ on 
$\fh^*$, and $\beta$ is an isomorphism
of $\zeta$ with $\cU (\fh)^ W$. Now, for 
$z \in \zeta$, 
$$
zv_ \lambda = \beta(z)(\lambda)v_ \lambda 
$$
but since $z$ commutes with the actions of all elements of $\cU (\fg)$ 
and so, as $v_ \lambda$ is cyclic
for $V_ \lambda$, we find that 
$zv=\beta(z)(\lambda)v$ on all  $v\in V_ \lambda $. Thus
$$
z = \chi_ \lambda (z)I,\quad  \hbox{on}\quad V_ \lambda , \qquad
\chi_ \lambda (z) := \beta(z)(\lambda).
$$
It follows from this that, if $\lambda, \mu \in \fh^*$ 
are such that $\chi_ \lambda (z) = \chi_ \mu (z)$ for all
$z \in \zeta$, then for some $s \in W $, 
we must have $s_ A \lambda = \mu$. A consequence of this
is the following: if $U_ 1 $, $U_ 2$
are highest weight modules with highest weights
$\mu_ 1 $, $\mu_ 2$ respectively, 
and if there is a non zero morphism $U_ 2 \lra U_ 1 $, then
there is an element $s \in W$ such that
$\mu_ 2 + \delta = s(\mu_ 1 + \delta)$. Indeed,
if $z \in \zeta$, then $z$ acts as $\chi_{ \mu_ 2} (z)$ on $U_ 2$ and by 
$\chi_ {\mu_ 1} (z)$
on $U_ 1$ and
these two numbers must be the same for all $z \in \zeta$.
This gives the required
result.
The Harish-Chandra homomorphism on $\zeta$ depends on the choice of
the positive system $P $. Let
$$
\gamma(z)(\lambda) = \beta_P (z)(\lambda -\delta_ P ) \qquad
(z \in \zeta).
$$
It is then easy to show that $\gamma(z)$ is independent of $P $,
and that $\gamma$ is an
isomorphism of $\zeta$ with the algebra $\cU (\fh)^ W$
of all elements of $\cU (\fh)$ invariant
under the usual linear action of $W$ on $\fh$ (see \cite{knapp} Ch. VII).

\subsection{Converse to Theorem \ref{thm1}} 

We want to construct highest weight
$(\fg, \fh)$-modules when $\lambda$ 
satisfies the condition of Theorem \ref{thm1}. In view of the
splitting $\fg= \fg_t \times \fg_1$,
 it is enough to consider the case of $\fg_t$, since we can
tensor with finite dimensional modules for $\fg_1$.

Let $\lambda \in \fh^*$ be such that $\lambda(H_\alpha )$
is an integer $\geq 0$ for all $\alpha \in P_k$. Let $F = F_\lambda$ 
be the irreducible finite
dimensional module for $\fk$ of highest weight $\lambda$. 
Note that $\lambda(H_\beta )$ can be
arbitrary for positive non-compact roots $\beta$. Write 
$\fq = \fk \oplus \fp^+$. Recall that
$[\fk, \fp^+ ] \subset \fp^+$ and 
so we can turn $F$ into a left $\fq$-module by letting $\fp^+$ act
trivially. 

Define
$$
U^\lambda = \cU (\fg) \otimes_{ \cU(\fq)} F
$$
and view $U^\lambda$ as a $\cU (\fg)$-module by left action:
$$
a(b \otimes f ) = ab \otimes f.
$$
Let
$$
\delta = (1/2)\sum_{\gamma\in P}
\gamma.
$$

\begin{theorem} \label{thm4} 
$U^\lambda$ is the universal HC module of highest weight $\lambda$. If
  $$
  (\lambda + \delta)(H_\gamma ) \quad \hbox{is real and} \quad \leq 0
  $$ for all 
$\gamma \in P_ n$, then $U^\lambda$ is irreducible.
\end{theorem}

\begin{proof} Let $M \subset U^\lambda$ 
be a nonzero submodule of $U^\lambda$. Since the weights of
$M$ are $\leq \lambda$ we can choose a maximal one, say $\mu$; 
if $u$ is a corresponding
weight vector, $X_ \gamma u = 0$ for all $\gamma \in P$ and 
so $\cU (\fg)u$ is a highest weight module of highest weight 
$\mu$. From the properties of the infinitesimal character
and the Harish-Chandra homomorphism this implies that 
$\mu + \delta = s(\lambda + \delta)$
for some $s \in W$.
The condition on $\lambda$ can be rewritten as 
$(\lambda + \delta)(H_\gamma )\geq 0 $ for all $\gamma \in
P = P_k \cup (-P_ n )$. Let $s_ 0$ be the element of $W_k$ 
that takes $P_k$ to -$P_k$.
Then $s_0 P = -P_k \cup P_n$  and so $P^ - = -s_0 P$ 
is also a positive system. Let
$\{\alpha_1 , \dots , \alpha_r , \beta_1 , \dots , \beta_r \}$ 
be the simple system of $P$
and let $\gamma_j = s_0 \beta_j$.
Then the $\gamma_j$ are also in $P_ t$ and in fact $\gamma_j$ 
is the highest root of the $\fk$-module
$\fp^+_j$ of which $\beta_j$ is the lowest root. Hence 
$\gamma_j = \beta_j + \sigma_j$ where $\sigma_j$ is a sum
of positive compact roots.

We claim that $\{\alpha_1 , \dots , \alpha_r , 
-\gamma_1 , \dots , -\gamma_r \}$ is the simple system of $P^-$.
Let $\gamma \in P_ n$. Then $s_0 \gamma = 
b_ 1 \alpha_1 + \dots + b_ r \alpha_r + e_ 1 \beta_1 + \dots + e_ s \beta_s$ 
where the
$b_ i$, $e_ j$ are integers $\geq 0$. 
Hence applying $-s_0$, we get $-\gamma = b_1 '  \alpha_1 + \dots +
b_ r ' \alpha_r - e_ 1 \gamma_1 - \dots -e_ s \gamma_s$ where the $b_ j'$ 
are integers $\geq 0$. This proves the claim.

Since $(\lambda + \delta)(H_\gamma )\geq 0$ for all 
$\gamma \in P^-$  we can write
$$
(\lambda + \delta) - s(\lambda + \delta) =\sum_{1\leq i\leq r}
c_ i \alpha_i - \sum_{1\leq j\leq s}
d_ j \gamma_j
$$
where $c_ i , d_ j\geq 0$. Thus
$$
\lambda - \mu = (\lambda + \delta) - s(\lambda + \delta) =
\sum_{1\leq i\leq r}
c_ i \alpha_i - \sum_{1\leq j\leq s}
d_ j \gamma_j
$$
But
$$
\lambda - \mu = \sum_{1\leq i\leq r}
a _i \alpha_i +
\sum_{1\leq j\leq s}
b_ j \beta_j
$$
where the $a_i , b j$ are integers
$\geq 0$. Hence, writing $\gamma_j = \beta_j + \sigma_j$ as above
and restricting to $\fc$ we get
$$
\sum_{1\leq j\leq s}
(d_ j+f_j) (\beta_j)|_\fc=0
$$
Since the $(\beta_j)|_\fc$ 
are linearly independent by (b) of Theorem \ref{thm3} we have 
$d_ j +f_ j =0$ for all $j$, and hence, as the $d_ j , f_ j$ are $\geq 0$, 
we must have $d_ j = f_ j =0$,
for all $j$. Hence
$$
\lambda - \mu =\sum_{1\leq i\leq r}
a_i \alpha_i
$$
where the $a_i$ are integers $\geq 0$.

\medskip
Now $u$ is a linear combination of 
$X_{ -\gamma_1} \dots X_{ -\gamma_m} v$ where each $\gamma_j$ is in
$\{\alpha_1 , \dots , \alpha_r ,$ $\beta_1 , \dots \beta_s \}$ and
$$
\lambda - \mu = \gamma_1 + \dots + \gamma_m.
$$
Writing each $\gamma_j$ as a linear combination of 
$\alpha_i$ $(1 \leq  i \leq  r)$ and the $\beta_j$ $(1 \leq 
j \leq  s)$ with integer coefficients $\geq 0$, and 
noting that $\lambda - \mu$ does not involve
the $\beta_j$, 
we conclude that each $\gamma_j$ does not involve any $\beta_j$. In other words,
$u \in \cU (\fk)v$. 
But then $u$ must be a multiple of $v$, showing that $M = U^\lambda$.
This proves that $U^\lambda$ is already irreducible.
\end{proof}

We shall now study the structure of $U^\lambda$ as a $\fq$-module for arbitrary
$\lambda$ with $\lambda(H_\alpha )$ an integer $\geq 0$, 
for all $\alpha \in P_k$. For this we need a standard
lemma.

\begin{lemma} \label{lemma4bis} 
Let $g$ be a field and $A$, $B$ algebras over $g$. Suppose $B \subset A$, $A$
is a free right $B$-module, $F$ a left $B$-module, and $V= A \otimes_B F$. If $(a_i )$
is a free basis for $A$ as a right $B$-module, and $L = \sum_i 
g.a_ i$, then the map
taking $l \otimes_ g f$ to $l \otimes_ B f$ is a linear isomorphism of 
$L \otimes_ g F$ with $V$.
\end{lemma}

\begin{proof} 
This is standard but we give a proof. All symbols $\otimes$ without any
suffix mean tensor products over the field $g$. Let $(b_ j )$ be a $g$-basis 
for $B$ with $b_ 0 = 1$. Then $V$ is a quotient of 
$A \otimes F$ by the span $S$ of elements
of the form $ab \otimes f - a \otimes bf$ where $a \in A$, 
$b \in B$, $f \in F$. Let $(f_ k )$ be a
$g$-basis for $F$. We assert that $S$ is spanned by 
$a_ i b_ j \otimes f_ k -a_ i \otimes b_ j f_ k$. Indeed,
$S$ is spanned by $a_ i b_ j b \otimes f_ k - a_ i b_ j \otimes bf_ k$. 
Now $$a_ i b_ j b \otimes f_ k - a_ i b_ j \otimes bf_ k =
(a_ i b_ j b \otimes f_ k - a_ i \otimes b_ j bf_ k ) -
(a_ i b_ j \otimes bf_ k - a_ i \otimes b_ j bf_ k ).$$
Expressing $b_ j b$ in the
terms of the first group as a linear combination of the $b r$ and the
$bf_ k$ of the
second group in terms of the $f_ l$, we see that our assertion is proved. Note
that we only need the terms with $j \neq 0$ as $a_ i b_ j \otimes f_ k 
- a_ i \otimes b_ j f_ k = 0$ for
$j = 0$. Since the map $L \otimes F \lra V$ 
is obviously surjective, it is enough to
show that the linear span of the 
$a_ i \otimes f_ k$ has $0$ intersection with $S$. Suppose
$\sum_{i,k} D_ {ik} a_ i \otimes f_ k \in S$. Then we can write
$$
\sum_{i,j,k, j\neq 0}
C_{ ijk} (a_i b_j \otimes f_k - a_i \otimes b_j f_k ) =
\sum_{i,k}
D_{ik} a_i \otimes f_k.
$$
Since $b_j f_k$ is a linear combination of the $f_r$ it follows that
$$
\sum_{i,j,k, j\neq 0}
C_{ ijk} a_i b_j \otimes f_k =
\sum
E_{ ir} a_i \otimes f_r.
$$
This means that $C_{ ijk} = 0$ for all $i,j,k$ with $j \neq 0$, hence that
$\sum_{i,k}
D_{ik} a_i \otimes f_k = 0$. This proves the lemma.
\end{proof}

We regard $\cU (\fp^ - ) \otimes F$ as a 
$\cU (\fp^ - )$-module by $a, b \otimes f \mapsto ab \otimes f$. Since
is stable under ad $\fk$ we may view $\cU (\fp^ - ) \otimes F$ as a 
$\fk$-module also.

\begin{corollary} 
The map $\phi : a \otimes f \mapsto a \otimes U(\fq) f$ 
is a linear isomorphism of
$\cU (\fp^ - ) \otimes F$ with $U^\lambda$ 
that intertwines the actions of $\cU (\fp^ - )$ and $\cU (\fk)$. In
particular, $U^\lambda$ is a free $\cU (\fp^ - )$-module with 
basis $1 \otimes \cU(\fq) f_ j$ where $(f_ j )$ is
a basis for $F$.
\end{corollary}

\begin{proof}
 Since $\fg =\fp^- \oplus \fq$ 
it follows that $a \otimes b \mapsto ab$ is a linear isomorphism
of $\cU (\fp^ - ) \otimes \cU (\fq)$ with $\cU (\fg)$. 
It is clear from this that $\cU (\fg)$ is a free right
$\cU (\fq)$-module, and that any basis of $\cU (\fp^ - )$ is a 
free right $\cU (\fq)$-basis for
$\cU (\fg)$. 

Lemma \ref{lemma4bis} now applies and shows that $\phi$ is an isomorphism. It
obviously commutes with the action of $\cU (\fp^ - )$. The verification of the
commutativity with respect to $\fk$ is also straightforward.
\end{proof}

\begin{remark}
This gives a formula for the multiplicity for the weight spaces
$U^\lambda [\mu]$ of $U^\lambda$. Let $\lambda_ 0 = \lambda$ and 
$\lambda_ i$, $(0 \leq  i \leq  r)$ be the weights of $F$ with $k_ i$ as
the multiplicity of $\lambda_ i$. Let $\gamma_ 1 , \dots , \gamma_ q$ 
be the distinct totally positive roots.
For any linear function $\nu$ on $\fh$ let 
$N(\nu)$ be the number of distinct ways of
writing $\nu = m_ 1 \gamma_ 1 + \dots + m_ q \gamma_ q$ where the 
$m_ j$ are integers $\geq 0$. Then
$$
\mathrm{dim} U^\lambda [\mu] =
\sum_{0 \leq  i \leq  r}
k_ i \, N(\mu_ i - \lambda).
$$
\end{remark}

\begin{remark} 
  There is a criterion for the Verma module $V^\lambda$ to be
  irreducible, namely that
$$
(\lambda + \delta)(H_\gamma ) \in
/ \{1, 2, \dots\} \qquad
\hbox{for all} \, \gamma \in P.
$$
This is due to M. Duflo \cite{duflo}
and it is a variant of the similar condition for the
spherical principal series for a complex group to be irreducible, due to K.
R. Parthasarathy, R. Ranga Rao and V. S. Varadarajan
\cite{prv}.
\end{remark}

\subsection{Totally positive roots, real HC pairs, and 
complex geometry}

In practice the HC pairs arise by complexification of real HC pairs. Let
$(\fg_0 , \fk_0 )$ be a real HC pair, $\fg_0$ simple. We assume  that $G_0$
is a connected real Lie
group with Lie algebra $\fg_0$ and $K_0$ is the analytic subgroup
defined by $\fk_0$.
We also assume that:
\begin{enumerate}
\item Ad $K_0$ is the maximal compact subgroup of Ad $G_0$; 
  \item $G_0$ and $K_0$ have the same rank:
    $$\fh_0 \subset \fk_0 \subset \fg_0$$

where $\fh_0$    is a CSA for both.
\end{enumerate}
    Let $A_0$
    be the Cartan subgroup
    (CSG) of $\fh_0$ in $G_0$ so that it is 
centralizer of $\fh_0$ in $G_0$. So Ad $A_0$ is
compact and the roots of $(\fg,\fk)$ are the eigencharacters of Ad $A_0$. 
Thus all roots are pure imaginary on $\fh_0$. 
Now, $\fg_0$ being a real form of $\fg$, there
is a conjugation $X \lra X^ \conj$ on $\fg$, 
which is an antilinear bijection of $\fg$
with itself preserving brackets, with $\fg_0$ as the set of its fixed points. 
Since it is the identity on $\fh_0$, it follows by conjugating 
$[H, X_ \beta ] = \beta(H)X_ \beta$ that $\fg_\beta^ \conj= \fg_{ -\beta}$.
In the above setting $(\fg_0 , \fk_0 )$ is a HC pair but $K_0$ need not
be semisimple, i.e., the center $\fc_0$ may be $\neq 0$. 
If $\fk_0$ has zero center, the group $K_0$
is compact even when $G_0$ is the simply connected group corresponding to
$\fg_0$ 
  and $G_0$ has finite center.

\begin{lemma} \label{lemma1bis}
We have $[\fp_0 , \fp_0 ] = \fk_0$ and $\fp_0$ 
is irreducible as a $\fk_0$ -module.
Moreover the $H_\beta$ for noncompact roots $\beta$ span $i\fh_0$ over $\R$.
\end{lemma}

\begin{proof}
The first statement follows from the fact that 
$[\fp_0 , \fp_0 ] \oplus \fp_0$ is a
nonzero ideal in $\fg_0$ and so has to be $\fg_0$. 
Indeed, since $\fp_0$ is stable under
$\fk_0$, so is $[\fp_0 , \fp_0 ]$, so that stability under $\fk_0$ is clear; 
the stability under $\fp_0$
is obvious. Since $[\fp, \fp] = \fk$ is spanned by 
$[X_\beta , X_\gamma ]$ for $\beta, \gamma$ noncompact,
it follows that $\fh$ is spanned by the $[X_\beta , X_{-\beta} ]$ 
for noncompact $\beta$ which
proves the last statement. 
Note that as all roots are pure imaginary on
$\fh_0$, $H_\gamma \in i\fh_0$ for all roots $\gamma$. 
The irreducibility of $\fp_0$ is however less trivial.
The Killing form is positive definite on $\fp_0$ 
and negative definite on $\fk_0$,
and is $\fg_0$ invariant. 
We write $\fp_0 = \oplus_{ 1\leq j\leq r}\fp_j$ where the $\fp_j$ are mutually
orthogonal and irreducible for $\fk_0$. Let $\fk_j := [\fp_ j ,\fp_j ]$, 
$\fg_ j := \fk_ j \oplus \fp_j$. We
shall prove that the $\fg_j$  are ideals in 
$\fg_0$ and that $\fg_0 = \oplus_{ 1\leq j\leq r} \fg_ j$. If this is
so, then $j = 1$ and we are done.
Fix $j$, $k$ with $1 \leq  j, k \leq  r$ and $j \neq k$. 
We shall prove the following in
succession.

\medskip\noindent
A. $[\fp_ j ,\fp_k ] =0$, $[\fk_j ,\fp_k ] =0$, $[\fk_j , \fk_k ] =0$, 
$j \neq k$.

\medskip\noindent
Let $X \in \fk_0$ , $Y_i \in\fp_i$, $i = j, k$. 
Then $B(X, [Y_j , Y_k ]) = B([X, Y_j ], Y_k ) =0$ as
$[X, Y_j ] \in \fp_j$. 
Since this is true for all $X \in \fk_0$, we must have $[Y_j , Y_k ] =0$.
If $Y_j$, $Z_j \in \fp_j$, $Y_k \in\fp_k$ and $X = [Y_j , Z_j ]$, 
then $[X, Y_k ] = -[[Z_j , Y_k ], Y_j ] - [[Y_k , Y_j ], Z_j ] =0$ 
by the previous result. 

If $Y_j$, $Z_j \in\fp_j$, $X_k \in \fk_k$, then
$$
[[Y_j , Z_j ], X_k ] = -[[Z_j , X_k ], Y_j ] - 
[[X_k , Y_j ], Z_j ] =0 
$$
by the previous result.

\medskip\noindent
B. $\fk_j \perp \fk_k$

\medskip\noindent
Let $Y_j$, $Z_j \in \fp_j$, $X_k \in \fk_k$. 
Then $B([Y_j , Z_j ], X_k ) = -B(Z_j , [Y_j , X_k ]) =0$
since $[Y_j , X_k ] =0$. So 
$\fg_1 , \dots , \fg_ r$ 
are mutually orthogonal and $[\fg_ j , \fg_ k ] =0$ for $j \neq k$. 
Since $B$ is separately definite on $\fk_0$ and $\fp_0$, 
the $\fk_j$ and $\fp_j$ are linearly independent
among themselves, from which it follows easily that the $\fg_j$  are linearly
independent. 
We now claim that the $\fg_j$  are subalgebras. Since 
$$
[\fk_j ,\fp_j ] \subset [\fk,\fp_j ] \subset\fp_j
$$ 
and 
$[\fp_ j ,\fp_j ] \subset \fk_j$, 
it is enough to verify that $[\fk_j , \fk_j ] \subset \fk_j$; but
$[\fk_0,\fp_j ] \subset\fp_j$ so that $[\fk_0, \fk_j ] \subset \fk_j$, 
and $[\fk_j , \fk_j ] \subset [\fk_0 , \fk_j ] \subset \fk_j$. 
As $[\fp_0 , \fp_0 ] =\fk_0$, $[\fp_ j ,\fp_k ] = 0$ $(j \neq k)$, 
we have $\fk_0 = \fk_1 \oplus \dots \oplus \fk_r$, hence 
$\fg_0 = \fg_ 1 \oplus \dots \oplus \fg_ r$.
The $\fg_j$  are subalgebras and $[\fg_ j , \fg_ k ] =0$ for $j \neq k$. 
Hence the $\fg_j$  are ideals.
\end{proof}

\begin{lemma} \label{lemma2bis}
Let $\fg_0$ be simple. Then $\fc_0$ is either 0 or has dimension 1.
In the latter case we can find a $J \in \fc_0$, unique up to a sign, 
such that ad($J)^2 = -I$ on $\fp_0$. 
Moreover $\beta(iJ) = \pm 1$ for all non compact roots $\beta$.
\end{lemma}

\begin{proof}
Suppose $\fc_0\neq 0$. Since $\fc_0$ is compact, the action of $\fc_0$ on 
$\fp_0$ is completely reducible with pure imaginary eigenvalues. 
The eigenvalues are of the form 0 or $\pm im_ j$ where $m_j$ 
are real non zero linear functions on
$\fh_0$. The eigenspaces are stable under 
$\fk_0$ and so $\fp_0$ is a direct sum of $\fp_j$
where the $\fp_j$ are stable under $\fk_0$ and the eigenvalues on 
$\C \cdot \fp_ j$ are either
$0$ or $\pm im_ j$. By Lemma \ref{lemma1bis}, 
there is only one of the $\fp_j$. If 0 is an eigenvalue, then $\fc_0$ centralizes $\fp_0$, hence $\fg_0$.
This is not possible and so the
eigenvalues of $\fc_0$ are $\pm im$. If 
dim$(\fc_0 ) > 1$ we can find a nonzero $H \in \fc_0$
such that $m(H) = 0$ so that $H$ centralizes $\fg_0$, which is not possible. So
dim($\fc_0$ ) = 1. The existence of $J$ and its uniqueness are now obvious. Since
$J$ has eigenvalues $\pm i$ on $\fp$, and $\fp$ is spanned by the 
$X_\beta$ for non compact
roots $\beta$, we have $\beta(iJ) = \pm 1$ for all $\beta$.
\end{proof}

\begin{remark}The tangent space to $S_0 := G_0 /K_0$ at $\overline{1}$, the
image of 1 in the
canonical map $G_0 \lra G_0 /K_0$, is isomorphic to $\fp_0$ 
as a $K_0$ module. Thus $\pm J \in
\rEnd(\fp_0 )$ and satisfies $(\pm J)^ 2 = -I$ 
and is fixed under the adjoint action
of $K_0$ on $\fp_0$, hence defines $K_0$ -invariant sections 
$\pm J$ of the endomorphism
bundle of the tangent bundle of $S_0$, satisfying $(\pm J)^ 2 =-I$. Thus we
have defined two canonical almost complex structures on $S_0$. We shall see
later that these are actually integrable and so define $G_0$-invariant complex
structures on $S_0$.

We also notice that the two canonical complex structures are in bijection
with the admissible simple systems of $\fg$; this fact is true more
in general, for a non necessarily maximal compact $K_0$,
as detailed in \cite{alekseevski} and refs. therein. In this case, in fact,
the complex structures on $G_0/K_0$ are in bijection with simple systems
of generalized root systems, as studied by Kostant in \cite{kostant}
and later on in \cite{df}, \cite{df2}.
\end{remark}

We now state a result, which is a refinement of Theorem \ref{thm3} in
this context.

\begin{proposition} \label{lemma3bis}
Let the notation be as above.
  \begin{enumerate}
\item If $\fc_0 =0$ 
  there are no admissible positive systems.
  \item If dim($\fc_0$ ) =
1, there are exactly two admissible positive systems containing a given
positive system $P_k$ of compact roots. For a fixed choice of $J$ these are
$P^\pm  = P_k \cup \pm Q$ where $Q$ is the $P$ set of non compact 
roots taking the value 1 at $iJ$.
The subspaces $\fp^\pm  := \sum_{\beta\in\pm Q} \fg_ \beta$
are stable under $\fk$, abelian, and
irreducible.
\item
Finally, if 
$\{\alpha_1 , \dots , \alpha_{\ell-1 }\}$ 
are the simple roots in $P_k$ , we can
find a unique non compact root $\beta$ such that 
$\{\alpha_1 , \dots , \alpha_{\ell-1} , \beta\}$ is the set of
simple roots of $P^+$, $\{\alpha_1 , \dots , \alpha_{\ell-1} , -\beta\}$ 
being then the simple roots in $P^-$.
\end{enumerate}
\end{proposition}

\begin{proof} 
  (1)
  Assume first that $\fc_0 = 0$. Let $P$ 
be a positive system and let
$R \subset P$ be a non empty set such that 
$L = \sum_{\beta\in R} \fg_\beta$ is $\fk$-stable. If $W_k$ is
the subgroup of the Weyl group generated by the compact roots, it is then
clear that $R$, 
the set of weights for $\fk$ in $L$, is stable under $W_k$. 
Fix a $\beta \in R$
and let $\lambda = \sum_{s\in W_k} s \cdot \beta$. 
Then $\lambda$ is invariant under $W_k$, and the condition
$s_ \alpha\lambda = \lambda$ 
for all compact roots $\alpha$ gives $\lambda(H_\alpha ) =0$ 
for all compact roots $\alpha$,
so that $\lambda =0$, 
since the $H_\alpha$ span $\fh$. 
This is impossible since all the $s \cdot \beta$ are
$>0$ and so their sum cannot be 0.

(2) Let us now suppose that dim($\fc_0$) = 1. If $\alpha$ is a compact root and
$\beta \in \pm P_n$  it is immediate that $\beta + \alpha$ and
$\beta$ have the same value $\pm 1$ at $iJ$.
If $\beta+\alpha$ is a root, it must be non compact and so is
in $\pm P^n$. Hence $[X_\alpha , X_\beta ]$
is either 0 or in $\fp^\pm$, showing that $\fp^\pm$  are stable under
$\fk$. If $\beta, \gamma$ are both in
$\pm P_n$, $\beta + \gamma$ takes the value $\pm 2$ at $iJ$ and
so is not a root, showing that $\fp^\pm$ 
are abelian. Now $\fg^ \conj_\beta
= \fg -\beta$ and so $\fp^- = (\fp^+ )^ \conj$. So, if there is a proper
subspace $\fq$ of $\fp^+$ stable under $\fk$,
then $(\fq + \fq ^\conj ) \cap \fp_0$ is a proper subspace
of $\fp_0$ stable under $\fk_0$ ,
which is not possible since $\fp_0$ is irreducible. So $\fp^\pm$ 
are irreducible.

(3) If the positive system $P'$  contains $P_k$
and is admissible,
there is a non empty subset $R \subset P^n$  such that
$L := \sum_{\beta\in R} \fg_ \beta$ is stable
under $\fk$. If $\beta \in R$, then $\beta \in P^ \pm$  and so
$L \cap \fp^ \pm  \neq 0$ showing that $L$ contains
one of $\fp^\pm$, hence equal to one of them. Hence $P' =P^ \pm$.
To find the simple roots of $\fp^+$,
let $\beta$ be the lowest weight (relative to
$P_ k$) for $\fp^+$ as a $\fk$-module.
If $\gamma$ is a non compact root in $P_ n$, it is then clear
that $\gamma - \beta$ is a sum of positive compact roots and so
$\{\alpha_ 1 , \dots , \alpha_{ \ell-1} , \beta\}$ is
the simple system for $\fp^+$; changing $\beta$ to $-\beta$,
we get the simple system for
$P^ -$. The uniqueness of $\beta$ is obvious.
\end{proof}

\begin{remark}\label{cmplx-str}
  The relation $\fp = \fp^+ \oplus
  \fp^ - = \fp^+ \oplus (\fp^+ ) ^\conj$ shows that the action
  of $K_0$ on $\fp$ splits as $\sigma \oplus \sigma ^\conj$ where
  $\sigma$ is the irreducible action of $K_0$ on
  $\fp^+$. If $L \in \rEnd(\fp_0 )$ commutes with $K_0$,
  then $L$ is a scalar $\lambda$ on $\fp^+$
  and $\lambda ^\conj$ on $\fp^ -$.
  If now in addition we have $L^ 2 = -I$ it is immediate that
  $\lambda = \pm i$. Thus $L = \pm J$. This shows that the only
  $G_0$ -invariant almost
complex structures on $S_0 = G_0 /K_0$ are those defined by $\pm J$. This also
shows that the type of the $\R$-representation of $K_0$ on $\fp_0$ is $\R$,
namely, that
the commutator is inside $\C_\R$, the complex numbers viewed as a real algebra.
\end{remark}

\subsection{Unitarity}

We extend the homomorphism $\beta : \cU [0]\lra \cU (\fh)$ constructed
in Sec. \ref{hc-homo} to a linear map $\cU (\fg) \lra\cU (\fh)$
by making it 0 on $\cU [q]$ for
$q \neq 0$. Let $V$ be a $\fg$-module.
Then $V$ is said to be \textit{unitary} if there is a
positive definite hermitian 
product $( , )$ for $V$ such that
$$
(Xu, v) + (u, Xv) =0, \qquad
(u, v \in V, X \in \fg_0)
$$
It is a theorem of Harish-Chandra that if $V$ is 
a unitary HC module and $H$
is the completion of $V$ in the norm $||x|| = + \sqrt{(x, x)}$, 
and if the $\fk_0$ -action
on $V$ lifts to a $K_0$-action, 
there is a unitary representation of $G_0$ in $H$ such
that $V$ is the space of $K_0$ -finite vectors of $H$ as a $\fg$-module
(see \cite{hc}, \cite{knapp}). 

To discuss unitarity it is convenient to define the adjoint operation
directly on $\cU (\fg)$. The map $X \mapsto -X$ 
extends to an antiautomorphism of
$\cU (\fg_0 )$; it can then be extended to a conjugate linear antiautomorphism
of $\cU (\fg)$. 
It is denoted by $a \mapsto a^*$ and has the following properties: 
\begin{enumerate}
\item[(i)] $a^{**} = a$ 
\item[(ii)] $(ab)^* = b^* a^*$ 
\item[(iii)] $a^*$ is conjugate linear in $a$ 
\item[(iv)] $X^* = -X$ for
all $X \in \fg_0$. 
\end{enumerate}
It is uniquely determined by these requirements. The unitarity
condition is now
$$
(au, v) = (u, a^* v), \qquad (a \in \cU (\fg), u, v \in V ).
$$

\begin{lemma}\label{lemma1unit}
We can choose the root vectors $X_\gamma \in \fg_\gamma$ in such a way that
$[X_\gamma , X_{-\gamma} ] = H_\gamma$ and
$$
X_\gamma^*=\begin{cases}
X_{-\gamma}, & \gamma \, \hbox{compact}\\
-X_{-\gamma},&  \gamma \, \hbox{ non compact}\end{cases}
\qquad
X_\gamma ^\conj
=\begin{cases}
-X_{-\gamma}, & \gamma \, \hbox{compact}\\
X_{-\gamma},  &\gamma \, \hbox{ non compact}\end{cases}
$$
where $\conj$ is the conjugation of $\fg$ defined by $\fg_0$.
\end{lemma}

\begin{proof}
Let $0 \neq X_\gamma \in \fg_\gamma$ be arbitrary to start with, but satisfying
$[X_\gamma , X_{-\gamma} ] = H_\gamma$. The relation 
$[H, X_\gamma ] = \gamma(H)X_\gamma$ gives, on applying
$*$, $[H, X_\gamma^* ] = -\gamma(H)X_\gamma^* $, so that 
$X_\gamma^* = c(\gamma)X_{-\gamma}$. As $*$ is involutive we get
$c(\gamma) ^\conj c(-\gamma) = 1$, where $\conj$ denotes 
complex conjugation. On the other
hand, from the standard theory, 
$[X_\gamma , X_{-\gamma} ] = B(X_\gamma , X_{-\gamma} )H_\gamma '$ where $B$ is
the Killing form and $H_\gamma '$ is the image of $\gamma$ 
under the isomorphism $\fh^* \cong \fh$
induced by $B$. 
Hence $H_\gamma = B(X_\gamma , X_{-\gamma} )H_\gamma '$. 
As $H_\gamma = cH_\gamma '$ where $c >0$,
we see that $B(X_\gamma , X_{-\gamma} )$ is real and $>0$. 
On the other hand, we claim
that if $X \neq 0$ is in $\fg$, 
$B(X, X^*)$ is $>0$ or $<0$ according as $X \in \fk$ or $X \in \fp$.
To see this, write $X= Y + iZ$ where $Y, Z \in \fg_0$; then
$$
B(X, X^* ) = B(Y + iZ, -Y + iZ) = -B(Y, Y ) - B(Z, Z)
$$
which proves our claim 
(since $Y, Z \in \fk_0$ or $\fp_0$ according as $X \in \fk$ or $\fp$).
This proves the claim. But now 
$B(X_\gamma , X_\gamma^* ) = c(\gamma)B(X_\gamma , X_{-\gamma} )$, so that,
from our earlier remark we infer that $c(\gamma)$ is $>0$ or 
$<0$ according as $\gamma$ is
compact or non compact. 
In any case $c(\gamma)$ is real and so $c(\gamma)c(-\gamma) = 1$.
Write $X_\gamma ' = |c(\gamma)|^{ -1/2} X_\gamma$. Then 
$[X_\gamma ' , X_{-\gamma}] = H_\gamma$ still. On the other hand,
$$
\begin{array}{rl}
(X_\gamma')^* &= |c(\gamma)|^{-1/2}
X_\gamma^* = |c(\gamma)|^{-1/2}
c(\gamma)|c(-\gamma)|^{ 1/2} X_{-\gamma}'=
\\ \\
&= c(\gamma)|c(\gamma)|^{ -1} X_{-\gamma}'
=\mathrm{sgn} c(\gamma)X_{-\gamma}'
\end{array}
$$
This proves the first assertion. The second is immediate
from $X_\gamma ^\conj = -X_\gamma^*$,
    {by the very definitions}.
\end{proof}


  Let $\beta : \cU (\fg) \lra \cU (\fh)$ be the Harish-Chandra
  homomorphism as in Sec. \ref{hc-homo}.

\begin{theorem} \label{thm2unit}
Let $\lambda \in \fh^*$ and let $\pi_\lambda$ 
be the irreducible highest weight module
of highest weight $\lambda$. The $\pi_\lambda$ is unitary if and only 
if $\beta(a^* a)(\lambda) \geq 0$ for all
$a \in \cU (\fg)$. 
In particular it is necessary that:
$$\lambda(H_\gamma ) 
\geq 0, \quad
\hbox{for compact} \quad \gamma, \qquad
\lambda(H_\gamma )\leq  0 \quad \hbox{ for non compact} \quad \gamma.
  $$
\end{theorem}

\begin{proof}
Assume that $\pi_\lambda$ acting on $V$ is unitary and write $v$ for the highest
weight vector. Since, for $H \in \fh_0$ 
$$
\lambda(H)(v, v) = (Hv, v) = -(v, Hv) =
-\lambda(H) ^\conj (v, v), 
$$
it follows that $\lambda$ is pure imaginary on $\fh_0$. Hence all
weights are pure imaginary on $\fh_0$ (as the roots are so). A similar argument
shows that the weight spaces are mutually orthogonal. By definition of $\beta$
it follows that for $c \in \cU (\fg)$, we have 
$(cv, v) = \beta(c)(\lambda)(v, v)$. In particular,
$$
0 \leq  (av, av) = \beta(a^* a)(\lambda)(v, v) 
$$
from which we get the necessary part.

\medskip
With $X_\gamma$ as in Lemma \ref{lemma1unit}, 
we have 
$X_{-\gamma}^*X_{-\gamma} = \pm X_\gamma X_{-\gamma}$ according as $\gamma$ is
compact or not. Hence 
$$
\beta(X_{-\gamma} X_{-\gamma} ) = 
\beta(\pm X_\gamma X_{-\gamma} ) = \pm H_\gamma
$$
so that $\lambda(H_\gamma )$
is $\geq 0$ or $\leq 0$ according as $\gamma$ is compact or not.
For the sufficiency of the condition for unitarity we define $(a, b) =
\beta(b^* a)(\lambda)$ for $a, b \in \cU (\fg)$. 
Then $(a, b)$ is sesqui-linear and $(a, a)$ is real
and $\geq 0$ for all $a$. Hence $(a, b)$ is Hermitian symmetric and 
$(a, a) \geq 0$ for all $a$. 
Let $R$ be the radical of this Hermitian form: $a \in R$ if and
only if $(a, a) =0$ or equivalently $(a, b) =0$ for all 
$b \in \cU (\fg)$. Moreover
$(ca, b) = \beta(b^* ca) = (a, c^* b)$. 
We claim that $R$ has the following properties.
\begin{enumerate}
\item[(i)] $R$ is a left ideal.
\item[(ii)] $X_\gamma$ $(\gamma >0)$, $H - \lambda(H)$ $(H \in \fh_0 )$ 
are in $R$.
\end{enumerate}

Let us prove these assertions. Since $(1, 1) = 1$ we have 
$1 \not\in  R$. If $a \in R$
and $c \in \cU (\fg)$, $(ca, b) = (a, c^* b) =0$ for all $b$ and so 
$ca \in R$, proving (i).
For $\gamma >0$, $(X_\gamma , b) = \beta(b^* X_\gamma ) =0$ for all 
$b$ and so $X_\gamma \in R$. 
$$
(H -\lambda(H), H -
\lambda(H)) = (-H + \lambda(H))(H - \lambda(H))(\lambda) =0, 
$$
proving (ii). At this stage we
know that $W= \cU (\fg)/R$ is a $\cU (\fg)$-module of highest weight 
$\lambda$ and $1$ as
the corresponding weight vector, and that $W$ is unitary. 
We claim that $W$
is irreducible, hence  isomorphic to $\pi_\lambda$. 
Suppose $W' \neq W$ is a submodule. Then $W'^\perp$
is also a submodule. Since the weight spaces are mutually orthogonal and
finite dimensional it follows that $W= W' \oplus (W')^\perp$. 
Since $W' \neq W $, we
must have $1 \not\in W '$. As 
$$
W [\lambda] = W ' [\lambda] \oplus (W ')^\perp [\lambda], 
$$
we must have $1 \in W '^\perp$.
Hence $(W ')^\perp = W$, showing that $W ' = 0$.
\end{proof}


\begin{observation}
These remarks imply that
the only irreducible finite dimensional unitary module is the trivial 
representation. Suppose $\pi_ \lambda$ is finite dimensional and unitary. 
Then $\lambda(H_ \gamma ) \geq 0$
for all roots $\gamma > 0$. By Theorem \ref{thm2unit}, 
we then conclude that $\lambda(H_ \gamma ) =0$ for
all non compact roots $\gamma$. 
By Lemma \ref{lemma1bis} 
$\lambda =0$, hence 
$\pi_\lambda$ is the trivial one
dimensional representation.
\end{observation}

\begin{remark}
By using global methods Harish-Chandra proved that when
dim($\fc_0$) = 1 and $(\lambda + \delta)(H_\ga ) \leq  0$ 
for all non compact positive roots $ \gamma$,
then the module $U^ \lambda$ of 
Theorem \ref{thm3} 
is unitary. 
The full set of unitary highest
weight modules was later on determined by Enright, Howe, and Wallach
\cite{hw}, (see also \cite{schmidt},
\cite{kz} and refs. therein). A generalization to the super setting
is due to Jakobsen \cite{jakobsen} (see also \cite{cfv},
  \cite{cfv2} and refs. therein).
\end{remark}

\section{Representations of the group} \label{reps-sec}

\subsection{Geometry}

The objective now is to construct the representations
of the group $G_0$ that correspond to the highest weight HC modules
constructed in Sec. \ref{alg-sec}.
We shall eventually assume that rk($\fg_0$)=rk($\fk_0$) and that
dim($\fc_0$)=1. But initially we drop the condition on the center of $\fk_0$.
\begin{lemma}\label{lemma1ch2}
  Let $M$ be a Lie group (real or complex), $A_ 1$, $A_ 2$
  Lie subgroups with $\Lie(A_ 1)$ + $\Lie(A_ 2) = \Lie(M)$.
  Then $A_1A_2$ is open in $M$. If $A_1 \cap A_2 =
  \{1\}$, then $A_ i$ are closed and $a_ 1 , a_ 2 \mapsto a_ 1 a_ 2$
  is an analytic diffeomorphism of $A_1 \times A_2$ with $A_1 A_2$.
\end{lemma}

\begin{proof}
  Let $\fa_i$ = Lie$(A_i )$, $\fm$ = Lie($M$).
  If $f$ is the map $a_ 1$, $a_ 2 \mapsto a_ 1 a_ 2$ of
  $A_1 \times A_2$ into $M$, then $df$ is submersive at
  $(1, 1)$ and hence everywhere since
  $f$ intertwines the actions $(b_ 1 , b_ 2 )$:
  $(a_ 1 , a_ 2 ) \mapsto (b_ 1 a_ 1 , a_ 2 b_ 2 )$ and
  $(b_ 1 , b_ 2 ) : x \mapsto
  b_1 xb_ 2$. Thus $A_1 A_2 = f (A_1 \times A_2 )$ is open in $M$.
  For the second part, note
  that when $A_1 \cap A_2 = \{1\}$, $f$ is bijective and $\fm$
  is the direct sum $\fa_ 1 \oplus \fa_ 2$
  so that $df$ is also bijective; $f$ is thus a diffeomorphism.
  So the $A_ i$ are
  closed in $A_1 A_2 $, i.e., locally closed.
  This means that they are open in their
closures, and being subgroups, are therefore closed.
\end{proof}

Let $G$ be a complex semisimple Lie group, $G_0 \subset G$ a connected real
form. Let $R_0 \subset G_0$ a closed subgroup and let
$$
\Lie(G_0 ) = \fg_0, \qquad
\Lie(G) = \fg \cong \C \otimes_\R \fg_0, \qquad
\Lie(R_0 ) = \fr_0.
$$

\begin{lemma} \label{lemma2ch2}
Suppose there exists a complex Lie subalgebra $\fq \subset \fg$ such that
$\fg_0 + \fq = \fg$, $\fg_0 \cap \fq = \fr_0$.
Assume that the subgroup $Q \subset G$ defined by $\fq$ is
closed. Then
\begin{itemize}
\item[(a)] $G_0 Q$ is open in $G$.
\item[(b)] If $R_ 1 = Q \cap G_0$, then $G_0 /R_ 1\cong G_0 Q/Q$.
\item[(c)] $G_0 /R_0$ has a $G_0$-invariant complex structure.
\end{itemize}
\end{lemma}

\begin{proof}
  (a) follows from Lemma \ref{lemma1ch2}. The map
  $g, q \mapsto gq$ of $G_0 \times Q \lra G$
  induces a $G_0 $-equivariant map of $G_0 /R_ 1 \cong G_0 Q/Q$
  with bijective differential.
  The complex structure on $G_0 /R_ 1$ is the pull back from $G_0 Q/Q$. But
  $R_0$ and $R_ 1$ have the same Lie algebra, namely $\fr_0$, and so
  $G_0 /R_0 \lra G _0 /R_ 1$
is a covering map. So the complex structure on $G_0 /R_ 1$ can be pulled back
to $G_0 /R_0$.
\end{proof}

From now on we assume that $\fg_0$ is semisimple and use notation of Sec.
\ref{alg-sec}.
We take $G$ to be simply connected. At first we assume only that $\fg_0$ and
$\fk_0$ have the same rank and that $\fh_0 \subset \fk_0$ is a CSA for
both $\fg_0$ and $\fk_0$. Fix
a positive system $P$ of roots for $(\fg,\fk)$.
Define
$$
\fb^ \pm  = \fh \oplus \sum_{\al \in \pm P}
\fg_ \alpha , \qquad
\fn^ \pm  = \oplus_{\al \in \pm P}  \fg_ \alpha.
$$

Let $B^ \pm$, $N^ \pm$  be the subgroups of $G$ defined by
$\fb^ \pm$, $\fn^ \pm$
We will drop the suffix ``+'' when clear
  from the context. Let $A$ (resp.
$A_ 0$) be the subgroup of $G$ (resp. $G_ 0$) 
defined by $\fh$ (resp. $\fh_ 0$). It follows
from Lemma \ref{lemma2ch2} that 
$\Gamma = N^-  B^+ = N^- AN^+$ 
is open in $G$ (the big Bruhat cell). Let
$\fu = \fk_ 0 \oplus i\fp_ 0$.
Then $\fu$ is a compact form of $\fg$ and the corresponding subgroup 
$U$ of $G$ is
compact and simply connected.
The subgroup $A_ 0$ defined by $\fh_ 0$ is a torus
and is a maximal torus of $U$ as well as $G_ 0$.

\begin{lemma}\label{lemma3ch2}
If $\fm_0$ is a real form of $\fg$ containing $\fh_0$, we have 
$\fm_0 + \fb^ + = \fg$.
In particular $M_0 B^ \pm$  is open in $G$, $M_0$ being the group 
defined by $\fm_0$.
\end{lemma}

\begin{proof}
If we denote the conjugation of $\fg$ with respect to $\fm_0$ 
by $X \mapsto X \widetilde{\,\,}$,
then 
$$
\fg\widetilde{\,\,\,}_\beta = \fg_{-\beta} \quad
\hbox{and} \quad 
X_{-\beta} = X_\beta\widetilde{\,\,} = (X_\beta + X_\beta \widetilde{\,\,}) 
- X_\beta \in \fm_0 + \fn ,
$$
 showing
that $\fn^ - \subset \fm_0 + \fn^ +$. So 
$\fm_0 + \fb^ +$ contains $\fg = \fn^ + + \fh + \fn ^-$. The second
assertion now follow from Lemma \ref{lemma2ch2}.
\end{proof}

\begin{corollary}
The conjugation of $\fg$ with respect to $\fm_0$ takes 
$\fg_ \beta$ to $\fg_{ -\beta}$ for
all roots $\beta$.
\end{corollary}

\begin{proof} Follows from the 
  {previous} proof. \end{proof}

If we regard $\fg$ as a Lie algebra over the reals, it is semisimple and
$\fg = \fu + i\fu$ is its Cartan decomposition. Then 
$\fg = \fu + i\fh_0 + \fn^ \pm$  are
its Iwasawa decompositions. 
Thus $G = U A_\R N^ \pm$  are the global Iwasawa
decompositions of $G$, where we write $A_\R$ for the subgroup defined by 
$i\fh_0$.
The exponential maps:
$$
\exp : \fn^ \pm  \lra N^ \pm  ,\qquad
i\fh \lra A_\R
$$
are analytic diffeomorphisms (resp. complex and real). Also $A$ normalizes
$N^ \pm $. We write $\conj$ 
for the conjugation of $\fg$ and $G$ with respect to $G_0$ and
$\fg_0$ respectively.

\begin{lemma} \label{lemma4ch2}
We have
$$
N^ - \cap B^ + = \{1\}, \qquad
A \cap N^ + = \{1\}.$$
In particular the map
$$\psi : N^ - \times A \times N^+ \lra G,
\qquad \psi(n^ - , h, n^ + ) = n^ - hn^ +$$
is an analytic diffeomorphism onto the 
big Bruhat cell $\GAMMA$, and $A$, $N^ \pm$  are
closed.
\end{lemma}

\begin{proof}
If $\exp Z \in B^ +$ for some $Z \in \fn^ - $, then $e^{ \ad(Z)}$ 
normalizes $\fb^ + $, hence
$[Z, \fb^+ ] \subset \fb^+$ which is impossible unless $Z = 0$. 
If $\exp Z \in A$ for some
$Z \in \fn^+ $, $e^{ \ad(Z)}$ is semisimple and unipotent, 
hence ad$(Z) = 0$, or $Z = 0$.
\end{proof}

\begin{lemma}\label{lemma5ch2}
  We have
  $$
  G_0 \cap B^ \pm  = A_0
  $$
  and
  $G_0 B^ \pm$  are open in $G$ while $G =U B^+$.
  Moreover $G_0 /A_0 \cong G_0 B^ \pm  /B^ \pm$  acquires a 
$G_0$ -invariant complex
structure. Finally, $N^ -$ may be viewed as a section for 
$\GAMMA/B^+ $, and the left
action of $A$ on $\GAMMA/B^+$ is given 
by
$h, nB^+ \mapsto (hnh^{ -1} )B^+$.
\end{lemma}

\begin{proof}
Let $a = hn^+ \in G_0$ where $h \in A$, $n^+ \in B^+$. Then $a = h ^\conj n^ -$
where $n^ - = n ^\conj \in N^ -$. So $hn^+ = h ^\conj n^ -$ 
giving $(h ^\conj )^{ -1} hn^+ = n^ -$.
Thus $n^ - = n^+ = 1$, $h = h ^\conj$ which give $a \in A_0$. 
The rest follow from
Lemmas \ref{lemma2ch2} and 
\ref{lemma3ch2}. Since $B^+$ is closed and $U$ is compact, $U B^+$ is closed,
and as it is also open, it is $G$. The last statement is trivial.
\end{proof}

\subsection{ Holomorphically induced sheaves and group representations}

The characters of $A_0$ extend uniquely to holomorphic characters of $A$. 
Since $G$
is simply connected, these are precisely of the form 
$$
\chi_\la:\exp H \mapsto e^{ \lambda(H)}
$$
 where
$\lambda$ is integral, i.e., 
$\lambda(H_ \gamma )$ is an integer for all roots $\gamma$. 
 We write $\chi_ \lambda$ again for the character 
$hn^+ \mapsto \chi_ \lambda (h)$ of $B^+$.
Throughout what follows, $\lambda$ will be a fixed holomorphic character of $A$
and we shall suppress mentioning $\lambda$ whenever there is no confusion. 
Let $\pi$
be the natural map $G \lra G/B^+ =: X$. 
For fixed $\chi = \chi_ \lambda$ and any open
set $E \subset X$ we define $L(E)$ to be the linear space of all functions $f$ 
on $\pi^{ -1} (E)$ that are holomorphic and satisfy:
\beq\label{keydef}
f (ub) = f (u)\chi(b)\quad \hbox{for all} \quad b \in B, \,
u \in \pi^{ -1} (E)
\eeq
By abuse of notation we also write sometimes $L(\pi^{-1} (E))$ for $L(E)$. 
Then $L : E \lra L(E)$
is a sheaf on $X$. For any $E$, $L(E)$ is a Fr\'echet space in the topology
of uniform convergence on compact sets. This sheaf may be naturally
interpreted as the sheaf of sections of a holomorphic line bundle on $X$
canonically associated to $\chi$.
Let $G_0$ be a connected real form of $G$. Write
$S = G_0 B^+$.
The restriction of $L$ to $S/B^+$ (which is open in $X$) is a seaf on $S/B^+$.
The group $G_0$ acts naturally from the left on the space $L(S/B^+ )$ of global
sections of this sheaf. It is easy to verify that this
action gives a representation
of $G_0$ on $L(S/B^+ )$. Our goal is to study this representation by
analyzing the action of $A_0$. Note that at this stage we are not asserting
that this space is even $\neq 0$. For the definition of a representation of a
locally compact group on a Fr\'echet space, or more generally, a complete
locally convex space, see \cite{hcacta}
pp. 5-14 (for basic facts on Fr\'echet spaces see \cite{Treves}).
In our case the holomorphy implies that all vectors are $C^ \infty$, that
is the action of the group on them is smooth; we do
not use this fact however.

\medskip
Let $F$ be a Fr\'echet space, $H$ a Lie group, and $K \subset H$ a compact
subgroup of $H$. In our applications we will have $H = G_0$ and $K = A_0 $
or $K_0$.
Let $R$ a representation of $H$ in $F$. 
For any irreducible character $\tau$ of $K$ of
degree $d(\tau )$, we define the operator
$$
P (\tau ) = d(\tau )
\int_K \tau (k)^\conj R(k)dk \qquad
\int_K dk = 1.
$$
Then $P (\tau )$ is a continuous projection 
$F \lra F (\tau )$ where $F (\tau ) = P (\tau )F$.
$F(\tau )$ is a closed subspace of $F$. We have
$P (\tau )P (\tau ' ) = 0$,  $(\tau \neq \tau ' )$.
The $P (\tau )$ commute with the $K$-action, and further, commute with any
continuous endomorphism $E$ of $F$ that commutes with $K$. It is easy to
show that $F (\tau )$ is the algebraic linear span of all 
finite dimensional subspaces of $F$ which are stable under 
$K$ and on which $K$ acts irreducibly
according to a representation with character $\tau$. 
The $F (\tau )$ are linearly independent (see \cite{vsv3}).

We say that the representation $R$ is \textit{$K$-finite},
if each $F (\tau )$ is
finite dimensional. In this case we write $F^0 = F_ K^0 = \sum_\tau F (\tau )$.
Let $F^ \infty$ be the subspace of $C ^\infty$ vectors for $R$. 
Then, $F ^\infty$ is dense in
$F$. Thus, $F ^\infty (\tau ) := P (\tau )F ^\infty$ is dense in $F (\tau )$.
For any $f \in F$ let $f_ \tau = P (\tau )f$. 
Then, $\sum_\tau f_ \tau$ is called the Fourier
series of $f$. One knows (\cite{hcacta}, \cite{vsv3})
that, when $f \in F ^\infty $,
$$
f =
\sum_{\tau}
f_ \tau, \qquad (f \in F ^\infty )
$$
where the series converges absolutely.

\begin{lemma} \label{lemma1reps}
Suppose $F_0 \subset F$ is a dense subspace, with $F_0 = \sum_\tau L_ \tau$ where
the sum is algebraic and all $L_ \tau$ are 
finite dimensional with $L_ \tau \subset F (\tau )$
Then $L_ \tau = F (\tau )$ for all $\tau $, $F^0 =\sum_\tau F(\tau)$, 
and $F^0\subset F ^\infty$. Suppose
$L \subset K$ is a compact subgroup and $F$ is $L$-finite; 
then $F$ is $K$-finite and
$F_ K^0 = F_ L^0$.
\end{lemma}

\begin{proof}
Clearly 
$L_ \tau = P (\tau )F_0$ is dense
in $P (\tau )F = F (\tau )$. 
But $L_ \tau$ is closed in $F$ because it is finite dimensional,
so that $F (\tau ) = L_ \tau$. 
Since $F ^\infty$ is dense in $F $, the same argument shows
that $F ^\infty (\tau )$ is dense in $F (\tau )$, 
hence $F ^\infty (\tau ) = F (\tau )$. In particular $F_0 =
\sum_\tau F (\tau ) \subset F^\infty$.
For the last assertion let s be an irreducible representation of $K$ with
character $\tau$. The restriction of $s$ to $L$ 
will contain irreducible representation (inequivalent) 
$t_1 , \dots t_r  $ of $K$ 
with characters $\sigma_1 , \dots , \sigma_r$ respectively. It is
then clear that $F_ K (\tau ) \subset \sum_i F_ L (\sigma_i )$. 
This proves that $F_ K (\tau )$ is finite 
dimensional, hence that $F$ is $K$-finite, and that $F_ K^0
\subset F_ L^0$. To prove equality
here, let us consider $F_ L (\sigma)$ where $\sigma$ is the character 
of an irreducible representation $t$ of $L$. 
We know that $F_ L (\sigma) \subset
 F_ K^\infty$ and so each $f \in F_ L (\sigma)$ has
an absolutely convergent expansion $f = \sum_\tau f_ \tau$ 
where the $\tau$ vary over the
irreducible characters of $K$. 
It is enough to show that $P (\tau )F_ L (\sigma) = 0$ for
almost all $\tau $; 
for then by the Fourier expansion, $F_ L (\sigma)$ is contained in the
sum of finitely many $F (\tau )$ and so is contained in $F_ K^0$.
Indeed, suppose for
some $\tau$ we have $P (\tau )F_ L (\sigma) \neq 0$. 
Since $P (\tau )$ commutes with $K$, hence with
$L$, this means that $P (\tau )F_ L (\sigma) \subset F (\tau ) \cap F_ L 
(\sigma)$. Actually we have equality
here; for, if $v \in F (\tau ) \cap F_ L (\sigma)$, then $P (\tau )v = v$. 
Let $S(\tau ) = F (\tau ) \cap F_L (\sigma)$.
Then $S(\tau ) \neq 0$, and, when $\tau$ varies over 
the set of characters for which
$P (\tau )F_L (\sigma) \neq 0$, the $S(\tau )$ 
are non-zero linear subspaces of $F_L (\sigma)$ which
are linearly independent (as the $F (\tau )$ are linearly independent). Hence
the number of such $\tau$ cannot exceed dim $F_L (\sigma)$.
\end{proof}

\subsection
    {Action of $\cU(\fg)$ on holomorphic sections}
    
For any (second countable) complex manifold $M$, write $\Hol(M )$ for
the algebra of holomorphic functions on $M$ equipped with the topology of
uniform convergence on compact sets. Then $\Hol(M )$ is a Fr\'echet space
\cite{Treves}.
If $H$ is a locally compact group acting on $M$ via holomorphic
diffeomorphisms
of $M $, then the corresponding natural action of $H$ on $\Hol(M )$ is a
representation (see \cite{hcacta} 
Lemma 1). 
Let
$$\GAMMA = N^-B^+$$
be the big Bruhat cell. 
We begin by studying $F = L(\GAMMA)$. Although $G$ does
not act on $F $, $\cU (\fg)$ does, 
from the left, as explained below. For any $Z \in \fg$
we have the right invariant vector field $\partial_ r (Z)$ on $G$. 
Then for any open $W \subset G$ and any $f \in \Hol(W )$,
$$
-(\partial_ r (Z))f (u) =\left(
\frac{d}{dt}\right)_{t=0}f (\exp(-tZ)u)
\qquad (u \in W ).
$$
The map $Z \lra -\partial_ r (Z)$ 
is a Lie homomorphism of $\fg$ into the Lie algera
of holomorphic vector fields on $W$. So it extends to a representation $\ell$ of
$\cU (\fg)$ on $\Hol(W )$. We refer to this as the left action of $\cU (\fg)$ 
on $\Hol(W )$.
Similarly there is an action from the right using the elements of $\fg$ viewed
as left invariant vector fields:
$$
(\partial (Z))f (u) =\left(\frac{d}{dt}\right)_{t=0}f (u\exp(tZ) )
\qquad (u \in W ).
$$
For $u = X_1 X_2 \dots X_r \in \cU (\fg)$, $X_i \in \fg$, we have
$$
(\partial(u)f )(g) = 
\Big(
(\partial^r f /\partial t_ 1 \dots \partial t_ r )
(g \exp(t_ 1 X_1 ) \dots \exp(t_ r X_r )) \Big)_0
$$
where $(\dots)_ 0$ means the derivative is calculated at 
$t_ 1 = \dots = t_ r = 0$. In
contrast to this is
$$
(\ell(u)f )(g) = \Big((\partial^r f /\partial t_ 1 \dots \partial t_ r )
(\exp(-t_ r X_r ) \dots \exp(-t_ 1 X_1 )g) \Big)_0.
$$
Clearly $\partial$ and $\ell$ 
commute with each other. We shall determine $\ell$ explicitly
on $\Hol(N^- AN^+ )$ (Theorem \ref{thm5reps}).
For $g \in \Hol(N^- )$ we write $g\wt$ for the unique element of $F$ whose
restriction to $N^-$ is $g$, so that 
$g\wt (nb) = g(n)\chi(b)$ for $n \in N^- $, $b \in B^+$.
The correspondence $g \leftrightarrow g\wt$
 is a linear topological isomorphism between
$\Hol(N^- )$ and $F $, 
the inverse being the restriction map from $F$ to $\Hol(N^- )$.
Since $\exp : \fn^- \lra N^-$ is an analytic diffeomorphism, we can introduce
global coordinates $(t _\alpha )_{ \alpha\in P}$ on $N^-$ by
$$
t _\alpha (n) = y _\alpha , \qquad
n = \exp \left( \sum_{\alpha\in P}
y _\alpha X_{-\alpha} \right)
$$
We have
$$
t _\alpha (a^{-1} na) = t _\alpha (n)\chi _\alpha (a), \qquad
(a \in A_0 , n \in N^- ).
$$
The polynomials in the $t _\alpha$ form a dense subalgebra $\cP$ of 
$\Hol(N^- )$ and
$\tP$ is the corresponding dense subalgebra of $F$. 
If $i_ a$ $(a \in A_0 )$ is the map
$n \mapsto ana^{-1}$ of $N^- $, it is immediate that
$$
\ell_ a \tg  = \chi(a)^{-1} (i_ a g)\widetilde{\,}, \qquad
(\ell_ a \tg  = \tg \circ \ell_a^{-1},  \, \, i_ a g = g \circ i_ a^{-1} ).
$$
For $r = (r _\alpha )$ let $t^r = \prod 
t _\alpha^{r _\alpha}$. We have
$$
i_\alpha t^r=\prod
(t _\alpha \circ i_a^{-1})^{r_\al}=
\chi_{r^*} t_ r \qquad
\left(
r^* =
\sum_\alpha
r _\alpha \alpha \right).
$$
Then, we obtain the following explicit decompositions of $\cP$ and $\tP$ 
under $A_0 $:
$$
\cP = \oplus_ r \C.t^ r , \qquad i_ a t^ r = \chi_{ r^*} t_ r, \qquad
\cP_ \lambda \wt = \oplus_ r \C.(t^ r )\widetilde{\,}, \qquad
\ell_ a (t_ r )\widetilde{\,} = \chi_{-\lambda+r^*} (t_ r )\wt.
$$
The elements of $\cP$ are polynomials in the coordinates $t _\alpha$. One can
describe $\cP$ also intrinsically. If we identify $N^-$ with $\fn^-$ 
via the exponential
map, then $\cP$ is simply the space of polynomials on the vector space $\fn^-$.
By the Baker-Campbell-Hausdorff
(BCH) formula the multiplication in $\fn^- \cong N^-$ is given by
\beq \label{bchreps1}
X \cdot Y = p(X, Y ), \qquad
X^{-1} =-X
\eeq
where $p$ is a polynomial map $\fn^- \times \fn^- \lra \fn^-$. 
We can also write this as
\beq \label{bchreps1bis}
n_1 n _2 = m(n_1 , n_ 2 ), \qquad
n^{-1} = i(n)
\eeq
where $m$, $i$ are polynomial maps $N^- \times N^- \lra N^-$ and $N^- \lra N^-$
respectively. There is also another way of thinking about this. $G$ may
be viewed as a complex affine algebraic group in a unique manner (the
regular functions on $G$ form the algebra generated by the matrix elements
of holomorphic finite dimensional representations of $G$); $N^-$ is an algebraic
subgroup of $G$ whose underlying variety is the vector space $\fn^-$. The
polynomials on $N^-$ are thus the regular functions on $N^- $,
denoted as $\cO_{N^-} (N^- )$. 
Using either point of view, one can conclude that the adjoint
representation Ad of $N^-$ on $\fg$ is rational and so its matrix elements are
polynomials on $N^-$.
Let $D^+$ be the semi group in $\fh^*$ 
generated by the positive roots,
namely, the set of elements of the form $\sum_i m_ i \alpha_ i$ where the 
$m_ i$ are integers
$\geq 0$ and the $\alpha_ i$ are the simple roots in $P$. Then
$$
\cP = \oplus_{d\in D^+} \cP_ d , \qquad
\cP_ d = \oplus_{ r^* =d} \C.t^ r
$$
while
$$
\cP = \oplus_ {d\in D^+} \cP_ d\wt  ,\qquad
\cP _d\wt  = \oplus_{ r^* =d} \C.(t^ r )\wt.
$$

\begin{lemma}\label{lemma2reps}
We have $F (\tau ) \neq 0$ if and only if 
$\tau = \chi_{-\lambda}+d$ for some $d \in D^+$.
Moreover
\beq\label{rep3}
\hbox{dim} F (\lambda+ d) = \# 
\left\{r = (r_ \alpha ) \, \big| \,r_\alpha \geq 0, \,
\sum_\al r_ \alpha \alpha = d \right\}
\eeq
\end{lemma}

\begin{proof}
Since $\cP$ is dense in $\Hol(N^- )$, we see that $\cP_\lambda\wt$ 
is dense in $F_\lambda$ ($F=L(\Gamma)$, in (\ref{keydef}) with $E=\Gamma$).
The first statement now follows from Lemma \ref{lemma1reps}. 
The second statement is
obvious.
\end{proof}

We have used the fact that the algebra $\Pol(V )$ polynomials on a finite
dimensional complex vector space $V$ form a dense subspace of $\Hol(V )$.
This is in general not true for arbitrary open subdomains of $V$; such
domains where this is true are called Runge domains.

For our purpose it
is enough to know the following.

\begin{lemma}
Let $V$ be a vector space with a compact torus $T$ acting on
it. Suppose that, 
for any character $\tau$ of $T $,
the subspace of $\Pol(V )$ 
of all $p$ such that 
$p(t^{-1} v) = \tau (t)p(v)$ for all $v \in V$, is finite dimensional.
Then, any
open connected subset of $V$, which is $T$-invariant and contains the origin,
is a Runge domain.
\end{lemma}

\begin{proof}
  We may assume that $V = \C^ N$ with $T$-action
$$
t, (z_ 1 , \dots , z_N ) \mapsto (f_ 1 (t)z_1 , \dots , f _N (t)z_N ) 
$$
where the $f_ j$ are characters of
$T$. Let $U$ be an open connected subset of $\C^ N$ containing the origin and
stable under $T$. The action of $T$ induces an action on $\Hol(U )$ which is a
representation. It is enough to prove that the closure of $\Pol(\C^ N )$ contains
$\Hol(U ) ^\infty$. Since the Fourier series of any $f$ in 
$\Hol(U ) ^\infty$ converges to $f $, it
is enough to show that any eigenfunction of $T$ in $\Hol(U )$ is a polynomial.
Suppose $g \neq 0$ is in $\Hol(U )$ such that 
$g(t^{-1} (u)) = f (t)g(u)$ for all $t \in T$ and
$u \in U $, $f$ being 
a character of $T$. Since $0 \in U$, we can expand $g$ as a power
series $g(u) = \sum_r c_ r u^r$ where we write $r = (r_ i )$ 
$1\leq i\leq N $, $u = u_1 \dots u_ N $,
valid in a polydisk. Then $c_ r f^ r = c_ r f$ whenever $c_ r\neq 0$
So only the $r$
with $f = f^r$ appear in the expansion of $g$. We claim that there are only
finitely many such $r$; once this claim is proven we are done, because $g$
is a linear combination of the monomials $u_ r$ with $f^r= f $, hence $g$ is a
polynomial. To prove the claim, note that all such $u_ r$ are eigenfunctions
for $T$ for the eigencharacter $f $, and by assumption, there are only finitely
many of these. 
\end{proof}

Recall that $\GAMMA = N^- B^+$ and that
$$
anb = (n^ a )ab, \qquad (n^a=ana^{-1}, \, a \in A_0)
$$
Thus, $A_0$-invariant open connected subsets $U$ of $N^-$ 
correspond via $U \leftrightarrow 
U\wt = U B^+$ to connected open sets $U\wt = A_0 U \wt B^+$. From the remark
above we know that any such $U$ is a Runge domain.
Let $\Gamma^1 = A_0 \Gamma B^+$ 
be connected and open in $G$, 
$\Gamma^2 = (\Gamma^1 \cap \GAMMA)^ 0$ where
the superscript $0$ refers to
the connected component containing $1$. 
Clearly
$\Gamma^2 = U_2 B^+$ where $U_2$ is a connected $A_0$-invariant open subset 
of $N^-$ containing 1. By what we saw above, $U_2$ is a Runge domain in 
$N^-$. Let
$$
F^1 = L(\Gamma^1 ), \qquad F^2 = L(\Gamma^2 ), \qquad (\Gamma^1=A_0\Gamma B^+,
\,\, \Gamma^2 = (\Gamma^1 \cap \GAMMA)^ 0)
$$

\begin{lemma} \label{lemma3reps}
Let notation be as above. Then the restriction map $F^1 \lra F^2$
is a continuous injection and $F^1 (\tau )$ maps into $F^2 (\tau )$ 
for all characters $\tau$
of $A_0$. Moreover $F^ 2 (\tau ) = F (\tau )$.
\end{lemma}

\begin{proof}
  The first statement is just the principle of analytic continuation.
  The second follows from the fact that the polynomials are dense in
$\Hol(U_ 2 )$.
\end{proof}


\begin{lemma} \label{lemma4reps}
The left action of $\cU (\fg)$ on $L(W )$ for any open $W \subset \GAMMA$ leaves
$\tP$ invariant.
\end{lemma}

\begin{proof}
Let $Z \in \fg$. Then, for fixed $n \in N^-$ and $t \in \C$ sufficiently small,
we have $\exp(-tZ ^{n^{-1}} )n \in N^- B^+$ and so
$\exp(-tZ^ {n^{-1}} )nb \in N^- B^+$ for all
$b \in B^+$. Write
\beq\label{repslemma4}
\exp(-tZ)nb = n \exp(-tZ^{ n^{-1}} )b, \qquad
\exp(-tZ n ) = \nu(n, t)\beta(n, t)
\eeq
with
$$
\nu(n, t) \in N^- , \quad \nu(n, 0) = 1,  \quad 
\beta(n, t) \in B^+ ,  \quad \beta(n, 0) = 1.
$$
Hence, for any $f \in \cP$,
$$
(\exp(tZ)\cdot f\wt)(nb) = 
f\wt (\exp(-tZ)nb) = f\wt (n\nu(n, t)\beta(n, t)b)
$$
resulting in
$$
(\exp(tZ)\cdot f \wt )(nb) = f (n\nu(n, t))\chi(\beta(n, t))\chi(b).
$$
We now differentiate with respect to $t$ at $t = 0$. Let
$$
V (n) = (d/dt)_{ t=0} \nu(n, t) \in \fn^- , \qquad
W (n) = (d/dt)_{ t=0} \beta(n, t) \in \fb^+.
$$
Then
$$
(Z\cdot f \wt )(nb) = (V (n)f )(n)\chi(b)+ f (n)d\chi(W (n))\chi(b).
$$
To determine $V (n)$ and $W (n)$, we differentiate (\ref{repslemma4}) 
at $t = 0$ to get
$$
-Z^{ n^{-1}}
= V (n)+ W (n)
$$
so that $V (n)$ (resp. $W (n)$) is the projection of $-Z^{ n^{-1}}$ on 
$\fn^-$ (resp. $\fb^+ $)
corresponding to the direct sum 
$\fg = \fn^- \oplus \fb^+$. We have seen already that
the adjoint action of $N^-$ on $\fg$ is rational. 
Hence if $(Y_ j )$, $(H_ i )$, $(X_k )$ are
bases of $\fn^- $, $\fh$, $\fn^+$ respectively, then
$$
V (n) = \sum_j
f_ j (n)Y_ j , \qquad W (n) =
\sum_i h_ i (n)H_ i+ \sum_k g _k (n)X_k
$$
where $f_ j $, $h_ i $, $g _k$ are polynomial functions on $N^-$. Hence
$$
(Z\cdot f \wt )(nb) = g \wt (nb)
$$
where
$$
g(n) =
\sum_j
f_ j (n)(Y_ j f )(n) +
\sum_i 
\lambda(H_ i )h_ i (n).
$$
From (\ref{bchreps1}) and (\ref{bchreps1bis}) 
we know that the action of $Y_ j$ on $f$ is by a polynomial
differential operator of degree 1. Hence we get a polynomial differential
operator of degree 1, say $D_ Z $, such that
$$
Z\cdot f \wt = g \wt ,\qquad
g = D_Z f.
$$
But $D_Z f \in \cP$ for all $f \in \cP$. 
This proves that for any $f \in \cP$, $Z\cdot f \wt \in \tP$.
\end{proof}

\subsection{Pairings of $U(\fg)$ modules of holomorphic sections}

Let $M_1$ , $M_2$ be two modules for $\fg$. By a $\fg$-pairing between them we
mean a bilinear form $\langle \cdot, \cdot \rangle$ on $M_1 \times M_2$
 with the property that
$$
\langle Xm_1 , m_2 \rangle = \langle m_1 ,-Xm_2 \rangle \qquad
(m_ i \in M_ i , X \in \fg).
$$
Since the $M_ i$ are modules for $\cU (\fg)$ this implies that
$$
\langle X_1 \dots X_r m_1 , m_2 \rangle = 
\langle m_1 , (-1)^ r X_r \dots X_1 m_2 \rangle \qquad
(m_ i \in M_ i , X_j \in \fg).
$$
The map $X \lra -X $ of $\fg$ is an involutive anti-automorphism of $\fg$. It
extends uniquely to an involutive anti-automorphism $u \lra u^ T$ of $\cU(\fg)$.
The $\fg$-pairing requirement is equivalent to
$$
\langle um_1 , m_2 \rangle = \langle m_1 , u^ T m_2 \rangle
(m_ i \in M_ i , u \in \cU (\fg)).
$$
We refer to this as a $\cU(\fg)$-pairing also. The pairing is said to be 
non-singular if $\langle m_1 , m_2 \rangle = 0$ for all $m_2$ 
(resp. for all $m_1 $) 
implies that $m_1 = 0$ (resp. $m_2 = 0$).

\begin{theorem}\label{thm5reps}
There is a non-singular $\cU(\fg)$-pairing between $\cP \wt$ and the
Verma module $V_ \lambda$. Moreover every non-zero submodule of $\cP \wt$ 
contains the
element $1 \wt$ corresponding to the constant function $1 \in \cP$. 
In particular,
the submodule $\cI \wt$ of $\cP \wt$ 
generated by $1 \wt$ is irreducible and is the unique
irreducible submodule of $\cP \wt$. Finally, $\cI \wt$ 
is the unique irreducible module
of lowest weight $-\lambda$.
\end{theorem}

\begin{proof} It is clear that
$$
(\ell(u)f )(1) = (\partial(u^T )f )(1);
$$
this is seen by taking $u = X_1 \dots X_r $, $X_i \in \fg$. 
We now define, for
$u \in \cU(\fg)$, $f \in \cP\wt $,
\beq\label{reps4}
\langle f, u\rangle = (\partial(u)f )(1) = (\ell(u^T )f )(1)
\eeq
Then, for $c \in \cU(\fg)$,
$$
\begin{array}{rl}
\langle \ell(c)f, u\rangle &= (\partial(u)\ell(c)f )(1) = (\ell(c)\partial(u)f )(1)
= (\partial(c ^T )\partial(u)f )(1) = \\ \\
&=(\partial(c ^T u)f )(1)
= \langle f, c ^T u\rangle.
\end{array}
$$
We thus have a $\fg$-pairing between $\cP\wt$ and $\cU (\fg)$ 
where the latter is regarded
as a $\cU(\fg)$-module under left multiplication. For 
$f \in \cP\wt$ we have, for all
$w \in \GAMMA$,
$$
f (w \exp tX_\alpha ) = f (w)\quad (\alpha > 0), \qquad
f (w \exp tH) = e^{ t\lambda(H)} f (w)
$$
from which we get
$$
\partial(Z)f = 0 \qquad (Z = X_\alpha , H- \lambda(H)).
$$
If $M_ \lambda$ is the left ideal generated by the 
$X_\alpha$ $(\alpha > 0)$ and $H- \lambda(H)$ for
$H \in \fh$, then we have
$$
\langle f, u\rangle = 0 \qquad (u \in M_\lambda ).
$$
Hence $\langle\cdot, \cdot \rangle$ 
defines a $\fg$-pairing between 
$\cP\wt$ and $\cU(\fg)/M_\lambda = V_ \lambda$, the
Verma module. We shall
now show that this is a non-singular pairing. 
If $\langle f, u\rangle = 0$ for all $u \in \cU(\fg)$,
then $(\partial(u)f )(1) = 0$ for all 
$u \in \cU (\fg)$ which implies that $f = 0$. Conversely,
suppose that $\langle f, u\rangle = 0$ for all 
$f \in \cP\wt$. We wish to prove that $u \in M_ \lambda$.
Now we can write $u$ as
$$
u = v+ \mu
$$
where $\mu$ is in the enveloping algebra of $N^-$
{and $v\in M_\lambda$}.
Since $(\partial(v)f )(1) = 0$, we
have $(\partial(\mu)f )(1) = 0$ for all $f \in \cP\wt$. 
This means that $(\partial(\mu)f )(1) = 0$ for
all polynomials $g$ on $N^-$. It is elementary to show that this implies that
$\mu = 0$, proving that $u \in M_ \lambda$.

The remainder of the proof is a formal consequence of the existence of
the non-singular pairing. The functions $(t^ r ) \wt$ 
corresponding to the coordinate polynomials $t^ r$ on $N^-$ 
are weight vectors for the action of $\fh$ for the
weight $r- \lambda$. Hence $\cP\wt$ is a weight module 
with the multiplicities defined
by Lemma \ref{lemma2reps}. 
We shall prove that every non-zero $\ell$-invariant subspace $W$
of $\cP\wt$ contains the vector $1 \wt$ defined by the constant function 1 on 
$N^-$.
Now $W$ is a sum of weight spaces and if it does not contain $1 \wt $, then $W$
is contained in the sum of all weight spaces corresponding to the weights
$\lambda- r$ where $r = (r _\alpha )$ with some $r_ \alpha > 0$. 
Now 
$$
\langle Hm_1 , m_ 2 \rangle =-\langle m_1 , Hm_2 \rangle
$$
for all $H \in \fh$, $m_1 \in \cP\wt$, $m_2 \in V_ \lambda$. 
This shows that the weight space of $\cP\wt$
for the weight $\theta$ is orthogonal to the weight space of 
$V_ \lambda$ for the weight $\phi$
unless $\theta =-\phi$. Let $v$ 
be a non-zero vector of highest weight $\lambda$ in $V_ \lambda$. Since
$W \subset \cP$
is contained in the span of weights other than -$\lambda$, we have 
$\langle W, v\rangle = 0$.
Hence, for all $g \in \cU(\fg)$, $w \in W$ we have 
$\langle \ell(g)w, v\rangle = 0$. So $\langle w, g^ T v\rangle = 0$
for all $g \in \cU(\fg)$. 
But $v$ is cyclic for $V_ \lambda$ and so we have $\langle
w, V_ {\lambda} \rangle= 0$ for
all $w \in W$. This means that $W = 0$, contradicting the hypothesis that
$W \neq 0$.
Thus every non-zero submodule of $\cP\wt$ contains the submodule $\cI \wt$
generated by $1 \wt$. This submodule is then the unique irreducible submodule 
of $\cP\wt$.
The weights of $\cI \wt$ are of the form $-\lambda+d$,
 where $d$ is a positive integral
linear combination of the simple roots, and $1 \wt$ has weight -$\lambda$. It is then
clear that $1\wt$ is the lowest weight of $\cI\wt$. This fact, together with its
irreducibility, characterizes it uniquely.
Theorem \ref{thm5reps} is thus fully proved.
\end{proof}

\subsection{Representations of the real group}

We now come to the real group $G_0$.

\begin{theorem} \label{thm6reps}
  Let $S = G_0 B^+$ and $F^1 = L(S)$. Then:
  \begin{enumerate}
    \item $F^1 \neq 0$ if and only
if $F^1$ contains an element $\psi$ which is an analytic 
continuation of $1 \wt$ to
$S$, i.e., which coincides with $1 \wt$ on $S \cap \GAMMA$. 
In this case $\psi \in (F^1 ) ^\infty$.
\item If
$F^{{11}} = \mathrm{Cl}(\cU(\fg)\psi)$ (Fr\'echet closure), then $F^{11}$ 
is a Fr\'echet module for $G_0 $, is $K_0$-finite, and
its $K_0$-finite part, which is also its $A_0$-finite part, 
is the irreducible lowest
weight module of lowest weight -$\lambda$. 
\item In particular $\lambda(H_ \alpha )$ is an integer $\geq 0$
for all compact positive roots.
\end{enumerate}
  \end{theorem}

\begin{proof}
(1) $F^1$ is a Fr\'echet module for $G_0$. We use Lemma \ref{lemma3reps} 
with $\Gamma^1 = S$
and $\Gamma^2 = (S \cap \GAMMA)^ 0$ 
where 0 refers to the connected component containing
the unit element of $G$. 
The restriction map $F^1 \lra F^2$ is a continuous
injection which is $A_0$-equivariant. so that 
$F^1 (\tau ) \subset F^2 (\tau ) = F (\tau )$ for all
characters $\tau$ of $A_0$. Thus $F^1$ is $A_0$-finite, 
hence also $K_0$-finite, with $(F^1 )^ 0 =\sum_\tau
F^1(\tau )$ as the space of its $K_0$-finite vectors (cf. Lemma 
\ref{lemma1reps}). Suppose
first that $F^1 \neq 0$. The weights corresponding to the characters $\tau$ are
$-\lambda+ r^*$ (earlier notation) with $(t_r ) \wt$
 as the corresponding eigenfunction,
the weight $-\lambda$ corresponding to $1 \wt$. Suppose now that the weights in
$(F^1 )^0$ do not include $-\lambda$. 
Then the corresponding eigenfunctions are the
$(t_r ) \wt$ where $r = (r_\alpha ) =
0$, and so vanish at 1. Thus all elements of
$(F^1 )^0$ vanish at 1, 
hence all elements of $F^1$ vanish at 1 as $(F^1 )^0$  is dense
in $F^1$. By left translation by elements of $G_0$ all elements of $F^1$ vanish
everywhere on $S$, i.e., $F^1 = 0$. This contradicts the assumption that
$L(S) \neq 0$. Hence $F^1 (\tau_0 ) \neq 0$ 
for the trivial character $\tau_0$ of $A_0$. Since
$$
F^1 (\tau_0 ) \subset F^2 (\tau_0 ) = \C \, 1 \wt , 
$$
we see that $F^1 (\tau_0 )$ contains an element $\psi$ which
restricts to $1 \wt$ on $(S \cap \GAMMA)^ 0$. 
The converse is obvious; if $\psi$ extends $1 \wt $, then
$F^1 (\tau )^ 0 \neq 0$, hence $F^1 \neq 0$.

(2) The algebra $\cU (\fg)$ acts on $(F^1 )^0$ and the restriction map 
commutes with the actions of $\cU(\fg)$. Let 
$J = \cU(\fg)\psi$. Then $J$ injects
onto a non-zero submodule of $\cI \wt $, hence $J$ maps onto $\cI \wt$ 
isomorphically.
Thus $J$ is the lowest weight module for the weight $-\lambda$. 
Now $\psi$ is both $K_0$-finite and $\zeta$-finite, where 
$\zeta$ is the center of $\cU(\fg)$, the latter because $\cI \wt \cong J$
has an infinitesimal character (as all lowest (or highest) weight modules
have). Hence Theorems 11 and 12 of \cite{vsv3} pg. 312
apply to tell us
that $\mathrm{Cl}(J)$ is invariant under $G$, is a Fr\'echet
module whose $K_0$-finite part
is precisely $J$.

(3) The last statement is clear because $F^{11}$ 
is a Harish-Chandra
module and $\psi$ is killed certainly by the compact negative roots, showing
that $-\lambda(-H_ \alpha ) = \lambda(H_ \alpha ) \geq 0$ 
for all compact positive roots. This finishes
the proof of the theorem.
\end{proof}

The shortcoming of Theorem \ref{thm6reps} is that it does not tell us when 
$L(S) \neq
0$. For instance, suppose that the real form $G_0$ is actually the maximal
compact subgroup $U_0 $, i.e., $G_0 = U_0$.Then $S = U_0 B^+ = G$, so that 
$L=L(S)$ is the
space of global sections of the holomorphic bundle. Then the result above
is essentially the Borel-Weil-Bott theorem in degree 0. Indeed, the space
of global sections, say $H$, is finite dimensional and carries an action of
$\cU (\fg)$ from the left. If $\lambda$ is a dominant integral linear function, then for the
corresponding representation with highest weight $\lambda$ the matrix element
$a _{11} (g)$ in a weight basis $(v_ i )$, $ 0\leq i\leq N$ with 
$v_0$ as the highest weight vector is a
global section of the sheaf so that $L \neq 0$. Then we know that the module
$J$ generated by $a_ {11}$ is irreducible and has lowest weight
$-\lambda$. We claim
that $J = L$; otherwise we can find a complementary module $R$ such that
$L = J \oplus R$, contradicting the fact that $R$ must contain $a_ {11}$.

\subsection{Analytic continuation of $1 \wt$ when the positive system of roots
  is admissible}\label{hc-cell-sec}

 We shall now assume that the center $\fc_0$ of $\fk_0$ 
has dimension
1, and that $P$ is an admissible system of positive roots. From the theory
on the Lie algebra this means the following. We have 
$$
\fp = \fp^+ \oplus \fp^- ; 
$$
$\fp ^\pm$  are
stable under $\fk$ and are irreducible, or what comes to the same thing, stable
and irreducible under $K_0$ (see Sec \ref{alg-sec}).
Since $\fh_0$ is a CSA of $\fg_0$ also, 
it follows that
$\fp ^\pm$  are sums of root spaces. 
The positive system of roots is $P = P_ k \cup P_ n$
where $P_ k$ is a positive system of compact roots, and $P_ n$ 
is the set of roots
whose root spaces are contained in $\fp^+$. Thus
$$
\fp^\pm=\sum_{\beta\in P_ n}
\fg ^\pm _\beta.
$$
We know that the simple system for $P$ is of the form
$$\{\alpha_1 , \dots , \alpha_r, \beta\}$$
where $\{\alpha_1 , \dots , \alpha_ r \}$ 
form the simple system for the compact roots, and $\beta$
is non-compact. Moreover the positive non-compact roots are of the form
$$
\gamma = m_1 \alpha_1 + \dots + m_ r \alpha_ r + \beta
$$
where the $m_ i$ are integers $\geq 0$. In particular
$[\fp ^\pm  , \fp ^\pm  ] = 0$
(same sign in both).
Obviously $\fp ^\pm$  are ideals in $\fn ^\pm $. 
Let $P ^\pm$  be the complex analytic subgroups
defined by $\fp ^\pm $. Let
$$
\fn_k ^\pm 
= \sum_{\be \in P_k}
\fg ^\pm _\beta.
$$
Then
$$
\fn ^\pm _k = \fn^\pm \cap \fk.
$$
The subgroups of $G$ defined by $\fn ^\pm$ 
 are closed and
$N ^\pm  = N_ k ^\pm  P ^\pm $.
We can then set up the map
$$
f : P^- \times K \times P^+ \lra G
\qquad
f (p^- , k, p^+ ) = p^- kp^+.
$$


\begin{lemma}\label{hc-cell-lemma}
The set $\Omega = P^- KP^+$ is open in $G$ and the map $f$ is a complex
analytic diffeomorphism onto $\Omega$.
\end{lemma}

\begin{proof}
This is standard and is valid under much greater generality. First
we compute the differential of $f$. We have, after a standard calculation,
$$
df_ {(p_1 ,k,p_2) }(X_1 , Z, X_2 ) = (X_1 + Z^ k + X_2^{ k p_2} )^{ (kp_2 )^{-1}}
$$
from which it is clear $df$ is bijective everywhere. We must show that $f$
is one-one. This reduces to showing that $P^- \cap K P^+ = \{1\}$. 
Let $p_1 \in KP^+$. Since $p_1 = k p^+$ where $k \in K$, 
$p^+ \in P^+ $, and Ad($K$) leaves $\fp^+$
invariant, 
it follows that Ad($p_1$) leaves $\fp^+$ invariant. Now $P^-$ is nilpotent,
even abelian, and so we can write $p_1 = \exp X_1$ where 
$0 \neq X_1 \in \fp^-$. Then
$e^ {\ad(X_1)} = \Ad(p_1 )$ leaves $\fp^+$ invariant. 
Because ad($X_1 $) is nilpotent, it is a
polynomial in $e^ {\ad(X_1)}$. 
Hence ad($X_1 $) leaves $\fp^+$ invariant. Using a suitable
lexicographic ordering of the roots we can 
write $X_1 = c_1 X_{-\beta_1} +X_{-\beta_2} +\dots$
where $\beta_1 < \beta_2 < \dots$ are non-compact roots and $c_1 \neq 0$. 
Clearly 
$$
[X_1 , X_{\beta_1} ] =
-c_1 H_{ \beta_1} + Y \qquad \hbox{where} \,\, Y \in \fn^-. 
$$
Hence 
$$
\Ad(p_1 )(X_{\beta_ 1} ) = e^{ \ad(X_1 )} (X_{\beta_ 1}) = X_{\beta_1}-
c_1 H_{ \beta_1} + Y '  \qquad \hbox{where} \,\, Y ' \in \fn^-. 
$$
But this must be in $\fp^+$. Thus $H_ {\beta_1} \in \fn^- + \fn^ +$
which is impossible.
\end{proof}


$\Omega$ is the big cell corresponding to the parabolic subgroups 
$P^ \pm $, known also as the \textit{Harish-Chandra open cell}.
In particular, $K$, $P^ \pm$  are closed in $\Omega$, 
hence locally closed in $G$, hence closed
in $G$.
We write $\theta$ for the automorphism of $\fg$ whose fixed point set is 
$\fk$,
namely the Cartan involution. 
Let
$$
\fu = \fk_0 + i\fp_0.
$$
Then $\fu$ is a compact form of $\fg$. One knows that 
$$
\fg = \fu \oplus i\fh_0 \oplus \fn^ +
$$
is an Iwasawa decomposition of $\fg$ 
(viewed as a real Lie algebra), see \cite{knapp} ch. VI. Let $U$ and
$A^ +$ be the real analytic subgroups of $G$ defined by 
$\fu$ and i$\fh_0$ respectively.
Then $U$ is a maximal compact subgroup of $G$ and simply connected, $A^ +$
is a vector space, and
$G = U A^ + N^ +$
is the Iwasawa decomposition of $G$ (viewed as a real Lie group). We write
$\theta \wt$ and $\eta$ for the conjugations of $\fg$ with respect to $\fu$ and 
$\fg_0$ respectively.
Since $G$ is simply connected it follows that $\theta$, 
$\theta \wt $, and $\eta$ all lift to $G$; we
denote them by the same symbol. The actions of these on the Lie algebra
are given by the following table:

\medskip
\begin{center}
\begin{tabular}{c|c|c|c|c}
 & $\fk_0$ & $i \fk_0$ & $\fp_0$ & $i\fp_0$ \\
\hline
$\theta$ & $\id$ & $\id$ & $-\id$ & $-\id$ \\
$\theta\wt$ & $\id$ & $-\id$ & $-\id$ & $\id$ \\
$\eta$ & $\id$ & $-\id$ & $\id$ & $-\id$
\end{tabular}
\end{center}

\medskip
It follows from this that
$$
\theta \wt = \eta\theta
$$
both on $\fg$ and $G$. Finally, we have
$$
\gamma(\fn^ + ) = \fn^ - , \qquad
\gamma(N^ + ) = N^ - \qquad
(\gamma = \eta, \theta \wt ).
$$ 
This follows from Lemma \ref{lemma3ch2} and Lemma \ref{lemma1unit}. 
We are thus able to formulate the crucial lemma.

\begin{lemma}
We have
$$
G_0 B^+ \subset P^ - KP^+.
$$
\end{lemma}

\begin{proof}
  Let us denote:
  $$
  \fn_k^+=\Lie(N^+_k), \quad \fn_n^+=\Lie(N_n^+), \qquad
  (\fn_k^+=\sum_{\al>0 \, {\mathrm{compact}}} \fg_\al, \,
  \fn_n^+=\sum_{\al>0 \, {\mathrm{non compact}}} \fg_\al)
  $$
  
  Since $KP^+$ contains $AN_ k^ + P^+ =
  B ^+$, we have $KP^+ B^ + = KP^+$
and so it is enough to prove that
$G_0 \subset P^ - KP^+$.
Now $G_0 = K_0 \exp \fp_0$, 
while the left side is invariant by left multiplication
by $K_0$; indeed,
$$
K_0 P^ - K P^+ \subset KP^ - K P^+ = P^ - KK P^+ = P ^- K P^+.
$$
Hence it suffices to prove that 
$\exp \fp_0 \subset P^ - KP^+$. Let $q = \exp X$ where
$X\in \fp_0$. Write $p = \exp(X/2)$ so that $q = p^2$. Then by the Iwasawa
decomposition for $G$ we have $p = uan^+$ where 
$u \in U, a \in A^ + , n^ + \in N ^+$.
We apply $\theta \wt$ to this relation. 
We have $\theta \wt (u) = u$ while $\theta \wt = -\id$ on $i\fk_0 $,
so that $\theta \wt (a) = a^{ -1}$. 
Further we have observed that $\theta \wt$ takes $N^ +$ to $N^ -$.
On the other hand,
$$
p = \eta(p) = \theta \wt (\theta(p)) = \theta \wt (p^ {-1} )
$$
so that $\theta \wt (p) = p ^{-1}$. Hence we have
$$
p ^{-1} = ua ^{-1} n_1, \qquad
n_1 \in N^ -
$$
giving
$$
p = n_1 ^{-1}au^{-1}
$$
Multiplying this by $p = uan^ +$ we get
$$
\exp X = q = p^2 = n ^{-1}_1 a^2 n^+
$$
showing that $q \in N^ - AN^ +$. But
$$
N^ - AN + = P^ - N_ n^ - AN_ k^ + P^+ \subset P^ - KP^+
$$
proving what we want.
\end{proof}

To get our result on analytic continuation of $1 \wt$ we need some 
preparation. 
Let $\lambda$ be an integral linear function on $\fh$ which is dominant for
compact positive roots, namely $\lambda(H_ \alpha ) \geq 0$ 
for all compact positive roots
$\alpha$. We wish to examine when there is an irreducible holomorphic 
representation $\sigma_\lambda$ of $K$ with a highest weight 
$\chi_ \lambda $, i.e., $\sigma _\lambda$ admits a non-zero
vector $v$ such that
$$
\sigma_ \lambda (h)v = \chi_ \lambda (h)v \quad
(h \in A), \qquad
\sigma_ \lambda (X_\alpha )v = 0 \quad (\alpha > 0 \quad \hbox{and compact}).
$$
(Here we use the same symbol for a representation on the group and the
differentiated representation on the Lie algebra). We are interested in the
case when $\fk_0$ has a non-zero center $\fc_0$.
Let $\fk '_0$ be the derived algebra $[\fk_0 , \fk_0 ]$; 
then $\fk '_0$ is semisimple and $\fk_0 =
\fk_0' \oplus \fc_0$. 
As usual, we drop the suffix 0 when we complexify. Thus $\fc$ is the
center of $\fk$, $\fk ' := [\fk, \fk]$ is semisimple, $\fk = \fk ' 
\oplus \fc$, and 
$\fh ' := \fh \cap \fk '$ is a CSA
of $\fk '$. Also, 
$\fh = \fh \cap \fk ' \oplus \fc$. 
Since $K$ is the fixed point set of $\theta$, we know
that $K$ is an algebraic group and so is $C$, the connected component of the
center of $K$. We have $K = K ' C$; moreover 
$Z = K ' \cap C$ is finite since it
is contained in the center of $K '$. Note that $C \subset A$. We then have the
following lemma.

\begin{lemma}\label{lemma8reps}
Let $\chi_ \lambda$ 
be a holomorphic character of $A$ such that $\lambda$ is dominant 
for compact positive roots. 
For the existence of an
holomorphic finite dimensional representation of $K$ with highest weight 
$\lambda$ it is
necessary and sufficient that the representation of $\fk '$ 
whose highest weight
is the restriction of $\lambda$ to 
$\fh '$ lifts to a holomorphic representation of $K '$.
\end{lemma}

\begin{proof}
The condition is obviously necessary. We now prove its sufficiency.
Let $K_ 1$ be the universal covering group of $K '$. 
Clearly 
the representation
of $\fk '$ with highest weight $\lambda$ lifts to a 
holomorphic representation $\tau '$ of $K_1$.
We thus have a representation 
$\tau := \tau ' \otimes \chi_ \lambda$ of $K_1 \times C$. The map 
$k, c \mapsto kc$
from $K_1 \times C$ to $K$ is onto $K$ and we shall show that 
$\tau$ is trivial on the
kernel of this map. Indeed, if $(k, c)$ is in the kernel, 
we have $kc =1$ and so
$k = c ^{-1} \in \Z$. 
The representation $\tau '$ is a scalar on the center of $K '$ and so
$\tau ' (t) = a(t)1$ for $t \in \Z$. But 
$\Z \subset A \cap K '$ since $A \cap K '$ is a maximal torus
of $K '$ and $A \cap K '$ acts as the character 
$\chi_ \lambda$ on the highest weight vector.
Hence $a(t) = \chi_ \lambda (t)$ for $t \in Z$. But then
$$
\tau ((k, c)) = \tau ' (k)\chi_ \lambda (c) = \chi_\lambda (c) ^{-1} 
\chi_ \lambda (c) = 1.
$$
\end{proof}

Let $\sigma_ \lambda$ be the representation of $K$ thus defined. 
It is obvious that it is the
required one.
We say that $\lambda$ is of \textit{$K$-type} if the condition of the lemma is 
satisfied. In the notation of the lemma, $\lambda$ is of $K$-type if 
and only if the
representation $\tau_1$ is trivial on the kernel 
of the covering map $K_1 \lra K '$.

\begin{theorem} \label{thm9reps}
Suppose that dim($\fc_0$) = 1. Let $\lambda$ be integral and 
$\lambda(H_ \alpha )$ be
an integer $\geq 0$ 
for all compact positive roots. Assume that $\lambda$ is of $K$-type.
Then $1 \wt$ extends analytically to an element $\psi \in L(S)$. 
In particular $L(S) \neq 0$
and $(F_\lambda)^ {11}$ carries the 
representation whose $K_0$ -finite part is the irreducible
lowest weight representation of lowest weight $-\lambda$. 
Finally and conversely,
if $\lambda$ is not of $K$-type, then $L(S) = 0$.
\end{theorem}

\begin{proof}
We have a representation $\sigma = \sigma_ \lambda$ of $K$ 
with highest weight vector
$v$ of weight $\chi_ \lambda$. Since $P^+$ is normal in 
$KP^+$ and $KP^+ /P^+ \cong K$ , we can
view $\sigma$ as a representation of 
$Q^ + = KP^+$ trivial on $P^+$. We define the
holomorphic function $f$ on $K$ by
$$
R\sigma(k)v = f (k)v
$$
where $R$ is the unique projection on $\C v$ modulo the sum of the remaining
weight spaces. Extend $f$ to a holomorphic function $g$ on 
$\Omega = P ^- KP^+$ by
$$
g(p^ - kp^+ ) = f (k) \qquad (k \in K, p^ \pm  \in P^ \pm  ).
$$
We now claim that the restriction of $g$ to $S = G_0 B^ +$ is the analytic
continuation of $1 \wt$. Since $\Omega$ 
contains the big cell $\GAMMA = N^ - AN^ +$ it s enough
to show that the restriction of $g$ to $\GAMMA$ is just $1 \wt$. 
Fix $n^ \pm  \in N^ \pm  $, $h \in A$. As
$$
N ^\pm  = P ^\pm  N_ k ^\pm  = N_ k ^\pm  P ^\pm 
$$
we can write
$$
n^ - = p^ - n_k^-, \quad
n^ + = n^ +_k p^+\quad
(n_ k ^\pm  \in N_ k ^\pm  , p ^\pm  \in P ^\pm  )
$$
and so
$$
g(n^ - hn^ + ) = g(p^ - n^-_k hn_k^+ p^+ ) = 
f (n_ k^- hn_ k^+ ) = \chi_ \lambda (h)f (n^-_ k ).
$$
It is thus enough to verify that $f (n^-_k ) = 1$.
 But it is clear that $\sigma(n k )v \cong v$
modulo the sum of weight spaces of weight less than $\lambda$. Hence 
$f (n^ -_k )=1$
as we wanted to check. This proves the theorem.
\end{proof}

\begin{corollary}
Let $K_0'$ be the analytic subgroup of $K_0$ defined by 
$\fk '_0 := [\fk_0 , \fk_0 ]$.
If $K_0'$ is simply connected, 
Theorem \ref{thm9reps} is valid for all integral $ \lambda$ whih are
dominant with respect to the compact positive roots.
\end{corollary}

\begin{proof} 
Since $K_0 '$ is simply connected, and is a compact form of $K ' $, it
follows that $K '$ is simply connected. So, for any integral $\lambda$ dominant
with respect to the positive compact roots, the corresponding irreducible
representation lifts to $K '$. Lemma \ref{lemma8reps} 
now shows that this representation
extends to $K$.
To prove the last statement, assume that $L(S) \neq 0$. It is clear that
the $K_ 0 ' $-module 
generated by $\psi$ has the lowest weight representation of $\fk_0'$
as its infinitesimal representation (same argument as for $G_0 $). 
Call this $\pi_\lambda$.
So $\pi_{ -\lambda}$ lifts to a representation of $K_0 '$. 
Now $K_0 '$ is semisimple and is the
compact form of $K '$, and so $\pi_ {-\lambda}$ 
extends to a holomorphic representation
of $K '$.
 This is contragredient to the representation with highest weight $\lambda$
which then lifts to $K$. By Lemma \ref{lemma8reps} 
we conclude that $\lambda$ is of $K$-type.
\end{proof}

\begin{example}
If $G = \SU(n, 1)$, $K_0 '$ is $\SU(n)$ and so the above Corollary applies. If
$G = \SO(2, 2k)$. then $K_0'$ is $\SO(2k)$ and the condition of 
$\lambda$ being of $K$-type
is non-trivial.
\end{example}


\begin{thebibliography}{99}

\bibitem{alekseevski} D. V. Alekseevski
\textit{Flag manifolds}, 
Zbornik Radova (1997) 14, 3-35.


\bibitem{bourbaki} 
N. Bourbaki, {\em Lie groups and Lie algebras}, Ch. 4--6, Springer--Verlag, 
2008, 300 + xii pp.
  
\bibitem{ccf}  C. Carmeli, L. Caston, R. Fioresi, with an appendix by 
I.~Dimitrov,  {\it  Mathematical Foundation of Supersymmetry},  
EMS Ser.~Lect.~Math., European Math.~Soc., Zurich, 2011.

\bibitem{cfv}
Carmeli, R. Fioresi, V.S. Varadarajan,
\textit{Highest weight Harish-Chandra supermodules and their
geometric realizations}, Transf. Groups, 
25 (2020), no. 1, 33–80.

\bibitem{cfv2}
Carmeli, R. Fioresi, V.S. Varadarajan,
\textit{Unitary Harish-Chandra Representations of Lie supergroups}, 
JNCG to appear, https://arxiv.org/abs/2103.16131,
  2021.


\bibitem{duflo} M. Duflo
  \textit{
    Sur la classification des ideaux primitifs dans l’algebre
    enveloppante d'une algebre de Lie semi-simple}. Ann. of Math. (2) 105 
    (1977), no. 1, 107-120.
    
\bibitem{cf} M.-K. Chuah, R. Fioresi
\textit{Hermitian real forms of contragredient Lie superalgebras},
J. Algebra {\bf 437} (2015), 161-176.

\bibitem{df} I. Dimitrov, R. Fioresi
\textit{On Kostant root systems of Lie superalgebras},
J. Algebra {\bf 570} (2021), 678--701.

\bibitem{df2} I. Dimitrov, R. Fioresi
  \textit{Generalized Root Systems}, preprint.

\bibitem{ev} T. J. Enright, V. S. Varadarajan, 
\textit{On an Infinitesimal Characterization of the Discrete Series},
Annals of Mathematics 102 (1),  1-15, 1975.

\bibitem{enright} T. J. Enright,
\textit{On the Fundamental Series of a Real Semisimple Lie Algebra},
Annals of Mathematics, 110, 1-82, 1979.


\bibitem{gz}
I. M. Gelfand and A. Zelevinski,
 \textit{Models of representations of classical groups and their
hidden symmetries},
Funct. Anal. Appl. {\bf 18} (1984), 183-198.

\bibitem{gs2}
V. Guillemin and S. Sternberg,
 \textit{Geometric quantization and multiplicities of group representations},
Invent. Math. {\bf 67} (1982), 515-538.

\bibitem{hw} T. Enright, Howe R., Wallach, N.
  \textit{A classification of unitary highest weight
modules} in "Proceedings of the University of Utah Conference 1982'',
  (P.C. Trombi ed.) 40 97-143. Progress in Mathematics, Birkhäuser 1983.
  

\bibitem{hcII} Harish-Chandra, 
\textit{Representations of semisimple Lie groups. II},
Trans. Amer. Math. Soc. 76 (1954), 26-65. 

  
\bibitem{hc} Harish-Chandra, \textit{Representations of semi-simple Lie
groups IV, V, VI}. Amer. J. Math. no. 77, 743-777 (1955); no. 78, 1-41 and
564-628, (1956).

\bibitem{hc2} Harish-Chandra \textit{Integrable and Square-integrable
representations},
American Journal of Mathematics
Vol. 78, No. 3 (1956), pp. 564-628.

\bibitem{hcacta}
Harish-Chandra,
\textit{Discrete series for semisimple Lie groups. II.
Explicit determination of the characters},
Acta Math. {\bf 116} (1966), 1-111.

\bibitem{helgason} S. Helgason,
{\it Differential Geometry, Lie Groups, and Symmetric Spaces},
Springer, Graduate Studies in Mathematics,
2001.

\bibitem{howe-hc} R. Howe, \textit{Harish-Chandra. 1923 - 1983. A bibliographical
memoir}, National Academy of Sciences, New York, 2011.


\bibitem{humphreys}
J. E. Humphreys. \textit{Introduction to Lie Algebras 
and Representation Theory}.
  Springer--Verlag, 1972.

  
\bibitem{jakobsen} H. P. Jakobsen, {\it
The Full Set of Unitarizable Highest Weight Modules of 
Basic Classical Lie Superalgebras}, Memoirs AMS, 
532, 1994.

\bibitem{ka} V.~G.~Kac,  {\it Lie superalgebras},  
Adv.~Math.~{\bf 26}  (1977), 8-26.

\bibitem{kirillov}
 Kirillov, A. A., \textit{Unitary representations of nilpotent Lie groups}, 
Doklady Akademii Nauk SSSR, 138: 283–284,  1961.

\bibitem{knapp} A. W. Knapp, 
\textit{Lie Groups Beyond an Introduction}, 
 Progress in Mathematics, Vol. 140, Birkhauser, Basel, 2002.

 \bibitem{kz}
   A. W. Knapp and Gregg J. Zuckerman
   \textit{Classification of Irreducible Tempered Representations of Semisimple Groups}
Annals of Mathematics, Vol. 116, No. 2 (1982),
389-455.



\bibitem{kostant}
  B. Kostant, {\em Root systems for Levi factors and
    Borel--de Siebenthal theory}, 
{\em Symmetry and Spaces}, Progr. Math. {\bf{278}}, 129--152, Birkh\"{a}user 
Boston, Inc., Boston, MA, 2010. 

\bibitem{kostant2}
  B. Kostant, {\em Quantization and unitary representations}, 
Lecture Notes in Mathematics, 170, Springer, Berlin, 1970.



\bibitem{langlands-hc}
Langlands, R. P. \textit{Harish-Chandra. 11 October 1923 – 16 October 1983}, 
Biographical Memoirs of Fellows of the Royal Society. 31: 198-225, 1985. 



\bibitem{prv}
  Parthasarathy, K.R., Ranga Rao, R., Varadarajan, V.S.
  \textit{Representations of complex semisimple Lie groups and Lie algebras}.
  Ann. Math. 85, 383-429 (1967).


\bibitem{schmidt} Schmidt W.,
  \textit{On the characters of the discrete series
    (the Hermitian symmetric case)}, Invent.
Math. 30 (1975), 47-144.
  
\bibitem{st} R.~Steinberg,  {\it Lectures on Chevalley groups},  
Yale University, New Haven, Conn., 1968.


\bibitem{Treves} F. Tr{\`e}ves, 
\textit{Topological vector spaces, distributions and kernels},
{Academic Press, New York}, xvi+624, 1967.

\bibitem{vsv1}
V.~S.~Varadarajan,  \textit{Lie groups, Lie algebras, and their
  representations}.  Graduate Text in Mathematics.  Springer-Verlag, 
New York, 1984.



\bibitem{vsv3} V.~S.~Varadarajan,  {\it Harmonic analysis on real 
reductive groups }, Lecture Notes in Mathematics, Vol. 576,
Springer-Verlag, Berlin-New York 1977.

\bibitem{vsv-hc1} V.~S.~Varadarajan,  {\it 
Harish-Chandra} 
J. Indian Math. Soc. 6: 191-21, 1991.

\bibitem{vsv-hc2}
Varadarajan, V. S.,   S.  Helgason, G.  D.  Mostow, and R. P. Langlands,
\textit{In Memoriam. Reminiscences of Harish-Chandra}. Princeton, N.J.: 
Institute for Advanced Study, 1984.

\bibitem{wallach} N. R. Wallach, {\it
Harmonic Analysis on Homogeneous Spaces}. Dover Publications, 2018.

  


\bibitem{warner} G.~Warner,  {\it Harmonic analysis on semi-simple
  {L}ie groups. {I}},  Die Grundlehren der mathematischen Wissenschaften,
  Band 188,
  Springer-Verlag, New York-Heidelberg, 1972.


  
\end{thebibliography}
\end{document}